\documentclass[a4paper,12pt]{article}

\usepackage[left=2cm,right=2cm, top=2cm,bottom=3cm,bindingoffset=0cm]{geometry}

\usepackage{verbatim}
\usepackage{amsmath}
\usepackage{amsthm}
\usepackage{amssymb}
\usepackage{delarray}
\usepackage{cite}
\usepackage{hyperref}
\usepackage{mathrsfs}
\usepackage{tikz}
\usetikzlibrary{patterns}
\usepackage{caption}
\DeclareCaptionLabelSeparator{dot}{. }
\newcommand{\al}{\alpha}
\newcommand{\be}{\beta}
\newcommand{\ga}{\gamma}
\newcommand{\de}{\delta}
\newcommand{\la}{\lambda}

\newcommand{\eps}{\varepsilon}
\newcommand{\vv}{\varphi}

\theoremstyle{plain}

\numberwithin{equation}{section}

\newtheorem{thm}{Theorem}[section]
\newtheorem{lem}[thm]{Lemma}
\newtheorem{prop}[thm]{Proposition}
\newtheorem{cor}[thm]{Corollary}

\theoremstyle{definition}

\newtheorem{example}[thm]{Example}
\newtheorem{alg}[thm]{Procedure}

\newtheorem{prob}[thm]{Problem}
\newtheorem{df}[thm]{Definition}

\theoremstyle{remark}

\newtheorem{remark}[thm]{Remark}

 \allowdisplaybreaks

\begin{document}

\begin{center}
{\Large\bf Local solvability and stability of an inverse spectral\\[0.2cm] problem for higher-order differential operators}
\\[0.5cm]
{\bf Natalia P. Bondarenko}
\end{center}

\vspace{0.5cm}

{\bf Abstract.} In this paper, we for the first time prove local solvability and stability of an inverse spectral problem for higher-order ($n > 3$) differential operators with distribution coefficients. The inverse problem consists in the recovery of differential equation coefficients from $(n-1)$ spectra and the corresponding weight numbers. The proof method is constructive. It is based on the reduction of the nonlinear inverse problem to a linear equation in the Banach space of bounded infinite sequences. We prove that, under a small perturbation of the spectral data, the main equation remains uniquely solvable. Furthermore, we estimate the differences of the coefficients in the corresponding functional spaces.

\medskip


\medskip

  
\vspace{1cm}

\section{Introduction} \label{sec:intr}

This paper deals with the differential equation 
\begin{equation} \label{eqv}
y^{(n)} + p_{n-2}(x) y^{(n-2)} + p_{n-3}(x) y^{(n-3)} + \dots + p_1(x) y' + p_0(x) y = \lambda y, \quad x \in (0,1),
\end{equation}
where $n \ge 2$, $p_k$ are complex-valued functions, $p_k \in W_2^{k-1}[0,1]$, $k = \overline{0,n-2}$, and $\la$ is the spectral parameter. Recall that:
\begin{itemize}
\item For $s \ge 1$, $W_2^s[0,1]$ is the space of functions $f(x)$ whose derivatives $f^{(j)}(x)$ are absolutely continuous on $[0,1]$ for $j = \overline{0,s-1}$ and $f^{(s)} \in L_2[0,1]$.
\item $W_2^0[0,1] = L_2[0,1]$.
\item $W_2^{-1}[0,1]$ is the space of generalized functions (distributions) $f(x)$ whose antiderivatives $f^{(-1)}(x)$ belong to $L_2[0,1]$.
\end{itemize}

We study the inverse spectral problem that consists in the recovery of the coefficients $(p_k)_{k = 0}^{n-2}$ from the eigenvalues $\{ \la_{l,k} \}_{l \ge 1}$ and the weight numbers $ \{ \be_{l,k} \}_{l \ge 1}$ of the boundary value problems $\mathcal L_k$, $k = \overline{1,n-1}$, for equation \eqref{eqv} with the corresponding boundary conditions
\begin{equation} \label{bc}
y^{(j)}(0) = 0, \quad j = \overline{0,k-1}, \qquad y^{(s)}(1) = 0, \quad s = \overline{0,n-k-1}.
\end{equation}

The main goal of this paper is to prove local solvability and stability of the inverse problem. This is the first result of such kind for arbitrary-order differential operators with distribution coefficients.

\subsection{Historical background}

Spectral theory of linear ordinary differential operators has a fundamental significance for mathematics and applications. For $n = 2$, equation \eqref{eqv} turns into the Sturm-Liouville (one-dimensional Schr\"odinger) equation $y'' + p_0 y = \la y$, which models various processes in quantum and classical mechanics, material science, astrophysics, acoustics, electronics. The third-order differential equations are applied to describing thin membrane flow of viscous liquid \cite{BP96} and elastic beam vibrations \cite{Greg87}. The third-order spectral problems also arise in the integration of the nonlinear Boussinesq equation by the inverse scattering transform \cite{McK81}. The fourth-order and the sixth-order linear differential operators appear in geophysics \cite{Bar74} and in vibration theory \cite{Glad05, MZ13}. Therefore, the development of general mathematical methods for investigation of spectral problems for arbitrary-order linear differential operators is fundamentally important.  

Inverse problems of spectral analysis consist in the reconstruction of differential operators by their spectral characteristics. Such problems have been studied fairly completely for Sturm-Liouville operators $-y'' + q(x) y$ with regular (integrable) potentials $q(x)$ (see the monographs \cite{Lev84, PT87, FY01, Mar11, Krav20} and references therein) as well as with distribution potentials of class $W_2^{-1}$ (see, e.g., the papers \cite{HM-sd, HM-2sp, FIY08, SS10, Hryn11, Eckh14, SS14} and \cite{Bond21-tamkang} for a more extensive bibliography). The basic results for the inverse Sturm-Liouville problem were obtained by the method of Gelfand and Levitan \cite{GL51}, which is based on transformation operators.
However, inverse problems for differential operators of higher orders $n > 2$ are significantly more difficult for investigation, since the Gelfand-Levitan method does not work for them. Therefore, Yurko \cite{Yur92, Yur00, Yur02} has developed \textit{the method of spectral mappings}, which is based on the theory of analytic functions. The central place in this method is taken by the contour integration in the complex plane of the spectral parameter $\la$ of some meromorphic functions (spectral mappings), which were introduced by Leibenson \cite{Leib66, Leib71}. Applying the method of spectral mappings, Yurko has created the theory of inverse spectral problems for arbitrary order differential operators with regular coefficients and also with the Bessel-type singularities on a finite interval and on the half-line (see \cite{Yur92, Yur00, Yur02, Yur93, Yur95}). Another approach was developed by Beals et al \cite{Beals85, Beals88} for inverse scattering problems on the line, which are essentially different from the spectral problems on a finite interval.

In \cite{MS16}, Mirzoev and Shkalikov have proposed a regularization approach for higher-order differential equations with distribution coefficients. This has motivated a number of researchers to study solutions and spectral theory for such equations (see the bibliography in the recent papers \cite{KMS23, Bond23-mmas}).
Inverse spectral problems for higher-order differential operators with distribution coefficients have been investigated in \cite{Bond21, Bond23-mmas, Bond23-reg, Bond22-alg, Bond23-res}. In particular, the uniqueness theorems were proved in \cite{Bond21, Bond23-mmas, Bond23-reg}. The paper \cite{Bond22-alg} is concerned with a reconstruction approach, based on developing the ideas of the method of spectral mappings. In \cite{Bond23-res}, the necessary and sufficient conditions for the inverse problem solvability were obtained for a third-order differential equation.

In this paper, we focus on local solvability and stability of the inverse problem. These aspects for the Sturm-Liouville operators were studied in \cite{PT87, FY01, Mar11, Borg46, BK19, Hoch77, McL88, Kor04, MW05, HK11, SS14, BB17, XMY22, GMXA23, SS10, HM-sd, Hryn11} and many other papers. Local solvability has a fundamental significance in the inverse problem theory, especially for such problems, for which global solvability theorems are absent or contain hard-to-verify conditions. Stability is important for justification of numerical methods.
For the higher-order differential equation \eqref{eqv} on a finite interval, stability of the inverse problem in the uniform norm was proved by Yurko for regular coefficients $p_k \in W_1^k[0,1]$ (see Theorem 2.3.2 in \cite{Yur02}). Furthermore, local solvability and stability in the $L_2$-norm were formulated without proofs as Theorems 2.3.4 and 2.5.3 in \cite{Yur02}. For distribution coefficients, local solvability and stability theorems have been proved in \cite{HM-sd} for $n = 2$ and in \cite{Bond23-res} for $n = 3$. However, to the best of the author's knowledge, there were no results in this direction for $n \ge 4$.

It is worth mentioning that, for the Sturm-Liouville operators with distribution potentials of the classes $W_2^{\alpha}[0,1]$, $\alpha > -1$, Savchuk and Shkalikov \cite{SS10} obtained the uniform stability of the inverse spectral problem. This result was extended by Hryniv \cite{Hryn11} to $\alpha = -1$ by another method. Recently, the uniform stability of inverse problems was also proved for some nonlocal operators (see \cite{But21, BD22, Kuz23}). However, the approaches of the mentioned studies do not work for higher-order differential operators. Thus, the uniform stability for them is an open problem, which is not considered in this paper.

\subsection{Main results}

In this paper, we for the first time prove local solvability and stability of an inverse problem for equation \eqref{eqv} and, in addition, obtain some sufficient conditions for global solvability. To the best of the author's knowledge, these aspects were not considered in previous studies for arbitrary-order differential equations with distribution coefficients, so our results are fundamentally novel. In order to prove the main theorems, we derive reconstruction formulas for coefficients $(p_k)_{k = 0}^{n-2}$, which are also novel for the considered class of operators.

Let us formulate the inverse problem for equation \eqref{eqv} and the corresponding local solvability and stability theorem.

Consider the problems $\mathcal L_k$ given by \eqref{eqv} and \eqref{bc}. For each $k \in \{ 1, 2, \dots, n-1\}$, the spectrum of the problem $\mathcal L_k$ is a countable set of eigenvalues $\{ \la_{l,k} \}_{l \ge 1}$. They are supposed to be numbered counting with multiplicities according to the asymptotics obtained in \cite{Bond22-asympt}:
\begin{equation} \label{asymptla}
\la_{l,k} = (-1)^{n-k} \left( \frac{\pi}{\sin \frac{\pi k}{n}} (l + \theta_k + \varkappa_{l,k}) \right)^n, \quad l \ge 1, \, k = \overline{1,n-1},
\end{equation}
where $\theta_k$ are some constants independent of $(p_s)_{s = 0}^{n-2}$ and $\{ \varkappa_{l,k} \} \in l_2$.

The weight numbers $\{ \be_{l,k} \}_{l \ge 1\, k = \overline{1,n-1}}$ will be defined in Section~\ref{sec:sd}. Thus, the spectral data $\{ \la_{l,k}, \be_{l,k} \}_{l \ge 1, \, k = \overline{1,n-1}}$ are generated by the coefficients $p = (p_k)_{k = 0}^{n-2}$ of equation \eqref{eqv}.
In this paper, we confine ourselves to $p$ of some class $W$ with certain restrictions on the spectra. 

\begin{df} \label{def:W}
We say that $p = (p_k)_{k = 0}^{n-2}$ belongs to the class $W$ if 

\smallskip

\textbf{(W-1)} $p_k \in W_2^{k-1}[0,1]$, $k = \overline{0,n-2}$.

\smallskip

\textbf{(W-2)} For each $k \in \{ 1, 2, \dots, n-1\}$, the eigenvalues $\{ \la_{l,k} \}_{l \ge 1}$ are simple. 

\smallskip

\textbf{(W-3)} $\{ \la_{l,k} \}_{l \ge 1} \cap \{ \la_{l,k+1} \}_{l \ge 1} = \varnothing$ for $k = \overline{1,n-2}$.
\end{df}

Thus, we study the following inverse problem.

\begin{prob} \label{ip:main}
Given the spectral data $\{ \la_{l,k}, \be_{l,k} \}_{l \ge 1, \, k = \overline{1,n-1}}$, find the coefficients $p = (p_k)_{k = 0}^{n-2} \in W$.
\end{prob}

Problem~\ref{ip:main} generalizes the classical inverse Sturm-Liouville problem studied by Marchenko \cite{Mar11} and Gelfand and Levitan \cite{GL51} (see Example~\ref{ex:2}). The uniqueness for solution of Problem~\ref{ip:main} follows from the results of \cite{Bond22-alg, Bond23-reg} (see Section~\ref{sec:sd} for details). Note that, if the conditions (W-2) and (W-3) are violated, then the spectral data $\{ \la_{l,k}, \be_{l,k} \}_{l \ge 1, \, k = \overline{1,n-1}}$ do not uniquely specify the coefficients $p$ and additional spectral characteristics are needed. In particular, for multiple eigenvalues in the case $n = 2$ the generalized weight numbers were defined in \cite{But07, BSY13}. For $n = 3$ and $\la_{l,1} = \la_{l,2}$, the additional weight numbers $\ga_l$ were used in \cite{Bond23-res}. For higher orders $n$, the situation becomes much more complicated, so in this paper we confine ourselves to the class $W$. Anyway, in view of the eigenvalues asymptotics \eqref{asymptla}, the assumptions (W-2) and (W-3) hold for all sufficiently large indices $l$.

Along with $p$, we consider an analogous vector $\tilde p = (\tilde p_k)_{k = 0}^{n-2} \in W$. We agree that, if a symbol $\alpha$ denotes an object related to $p$, then the symbol $\tilde \al$ with tilde will denote the analogous object related to $\tilde p$. The main result of this paper is the following theorem on local solvability and stability of Problem~\ref{ip:main}.

\begin{thm} \label{thm:loc}
Let $\tilde p = (\tilde p_k)_{k = 0}^{n-2} \in W$ be fixed. Then, there exists $\de > 0$ (which depends on $\tilde p$) such that, for any complex numbers $\{ \la_{l,k}, \be_{l,k} \}_{l \ge 1, \, k = \overline{1,n-1}}$ satisfying the inequality 
\begin{equation} \label{condloc}
\Omega := \left( \sum_{l = 1}^{\infty} \Biggl( \sum_{k = 1}^{n-1} \bigl( l^{-1} |\la_{l,k} - \tilde \la_{l,k}| + l^{-2} |\be_{l,k} - \tilde \be_{l,k}| \bigr) \Biggr)^2 \right)^{1/2} \le \de,
\end{equation}
there exists a unique $p = (p_k)_{k = 0}^{n-2}$ with the spectral data $\{ \la_{l,k}, \be_{l,k} \}_{l \ge 1, \, k = \overline{1,n-1}}$. Moreover, 
\begin{equation} \label{estp1}
\| p_k - \tilde p_k \|_{W_2^{k-1}[0,1]} \le C \Omega, \quad k = \overline{0,n-2},
\end{equation}
where the constant $C$ depends only on $\tilde p$ and $\de$.
\end{thm}

Theorem~\ref{thm:loc} generalizes the previous results of \cite{HM-sd} for $n = 2$ and of \cite{Bond23-res} for $n = 3$. However, for $n \ge 4$, to the best of the author's knowledge, Theorem~\ref{thm:loc} is the first existence result for the inverse problem solution in the case of distribution coefficients.

The proof of Theorem~\ref{thm:loc} is based on the constructive approach of \cite{Bond22-alg}. Namely, we reduce the nonlinear inverse problem to the so-called \textit{main equation}, which is a linear equation in the Banach space of  bounded infinite sequences. The unique solvability of the main equation follows from the smallness of $\delta$. Furthermore, we derive reconstruction formulas for the coefficients $(p_k)_{k = 0}^{n-2}$ in the form of infinite series. The crucial step in the proof is establishing the convergence of those series in the corresponding spaces $W_2^{k-1}[0,1]$ (including the space of generalized functions $W_2^{-1}[0,1]$). In order to prove the convergence, we rigorously analyze the solution of the main equation and obtain the precise estimates for the Weyl solutions. Along with Theorem~\ref{thm:loc}, we also prove Theorem~\ref{thm:glob} on global solvability of the inverse problem under several requirements on an auxiliary model problem.

The paper is organized as follows. In Section~\ref{sec:prelim}, we discuss the regularization of equation \eqref{eqv} and provide other preliminaries. In Section~\ref{sec:sd}, the weight numbers are defined and the properties of the spectral data are described. In Section~\ref{sec:main}, we derive the main equation basing on the results of \cite{Bond22-alg}. In Section~\ref{sec:rec}, the reconstruction formulas for the coefficients $(p_k)_{k = 0}^{n-2}$ are obtained. In Section~\ref{sec:solve}, we prove solvability and stability of the inverse problem. Section~\ref{sec:concl} contains concluding remarks.

\section{Preliminaries} \label{sec:prelim}

In this section, we explain in which sense we understand equation \eqref{eqv}. For this purpose, an associated matrix and quasi-derivatives are introduced. In addition, we provide other preliminaries. We begin with some notations, which are used throughout the paper:

\begin{itemize}
\item $\de_{j,k}$ is the Kronecker delta.
\item $C_k^j = \frac{k!}{j!(k-j)!}$ are the binomial coefficients.
\item In estimates, the same symbol $C$ is used for various positive constants that do not depend on $x$, $l$, $\lambda$, etc.
\item The spaces $W_2^k[0,1]$ are equipped with the following norms:
\begin{align*}
\| y \|_{W_2^k[0,1]} & = \left( \sum_{j = 0}^k \| y^{(j)} \|_{L_2[0,1]}^2 \right)^{1/2}, \quad k \ge 0, \\
\| y \|_{W_2^{-1}[0,1]} & = \inf_{c \in \mathbb C} \| (y^{(-1)} + c) \|_{L_2[0,1]}.
\end{align*}
\end{itemize}

For $n \ge 3$, consider the differential expression
$$
\ell_n(y) := y^{(n)} + \sum_{k = 0}^{n-2} p_k(x) y, \quad x \in (0,1).
$$

Fix any function $\sigma \in L_2[0,1]$ such that $p_0 = \sigma'$. 
Define the associated matrix $F(x) = [f_{k,j}(x)]_{k,j =1}^n$ for the differential expression $\ell_n(y)$ by the formulas
\begin{equation} \label{entf}
f_{n-1,1} := -\sigma, \quad f_{n,2} := \sigma - p_1, \quad f_{n,k} := -p_{k-1}, \:\: k = \overline{3,n-1}.
\end{equation}
All the other entries $f_{k,j}$ are assumed to be zero. Clearly, $f_{k,j} \in L_2[0,1]$.

Using the matrix function $F(x)$, introduce the quasi-derivatives
\begin{equation} \label{quasi}
y^{[0]} := y, \quad y^{[k]} := (y^{[k-1]})' - \sum_{j = 1}^k f_{k,j} y^{[j-1]}, \quad k = \overline{1,n},
\end{equation}
and the domain
$$
\mathcal D_F := \{ y \colon y^{[k]} \in AC[0,1], \, k = \overline{0,n-1} \}.
$$

Due the the special structure of the associated matrix $F(x)$, we have
\begin{gather} \label{quas1}
y^{[k]} = y^{(k)}, \quad k = \overline{0, n-2}, \qquad
y^{[n-1]} = y^{(n-1)} + \sigma y, \\ \label{quas2}
y^{[n]} = (y^{[n-1]})' + \sum_{k = 1}^{n-2} p_k y^{(k)} + \sigma y',
\end{gather}
and so $\mathcal D_F \subset W_1^{n-1}[0,1]$.

Note that the differential expression is correctly defined in the sense of generalized functions for any $y \in W_1^{n-1}[0,1]$. However, if $y \in \mathcal D_F$, then $y^{[n]} \in L_1[0,1]$ and relations \eqref{quas1} and \eqref{quas2} directly imply the following lemma.

\begin{lem} \label{lem:reg}
For $y \in \mathcal D_F$, $\ell_n(y)$ is a regular generalized function and $\ell_n(y) = y^{[n]}$.
\end{lem}

Thus, for $y \in \mathcal D_F$, $\ell_n(y)$ is a function of $L_1[0,1]$ and the relation $\ell_n(y) = y^{[n]}$ gives the \textit{regularization} of this differential expression. We call a matrix function $F(x)$ \textit{an associated matrix} of the differential expression $\ell_n(y)$ if $F(x)$ defines the quasi-derivatives $y^{[k]}$ and the domain $\mathcal D_F$ so that the assertion of Lemma~\ref{lem:reg} holds. 
A function $y$ is called \textit{a solution} of equation \eqref{eqv} if $y \in \mathcal D_F$ and $\ell_n(y) = \la y$ a.e. on $(0,1)$.

Following the technique of \cite[Section 2]{Bond22-alg}, we consider along with $F(x)$ the matrix function $F^{\star}(x) = [f_{k,j}^{\star}(x)]_{k,j = 1}^n$ such that
$$
f_{k,j}^{\star}(x) = (-1)^{k+j+1} f_{n-j+1,n-k+1}(x), \quad k,j = \overline{1,n}.
$$

Using \eqref{entf}, we obtain
\begin{equation} \label{entfs}
f_{k,1}^{\star} = (-1)^{k+1} p_{n-k}, \:\: k = \overline{2,n-2}, \quad
f_{n-1,1}^{\star} = (-1)^n (p_1 - \sigma), \quad f_{n,2}^{\star} = (-1)^n \sigma,
\end{equation}
and all the other entries $f_{k,j}^{\star}$ equal zero. For example, for $n = 6$, we have
$$
F(x) = \begin{bmatrix}
            0 & 0 & 0 & 0 & 0 & 0\\
            0 & 0 & 0 & 0 & 0 & 0\\
            0 & 0 & 0 & 0 & 0 & 0\\
            0 & 0 & 0 & 0 & 0 & 0\\
            -\sigma & 0 & 0 & 0 & 0 & 0\\
            0 & \sigma - p_1 & -p_2 & -p_3 & -p_4 & 0
        \end{bmatrix}, \qquad
F^{\star}(x) = \begin{bmatrix}
            0 & 0 & 0 & 0 & 0 & 0\\
            -p_4 & 0 & 0 & 0 & 0 & 0\\
            p_3 & 0 & 0 & 0 & 0 & 0\\
            -p_2 & 0 & 0 & 0 & 0 & 0\\
            p_1 - \sigma & 0 & 0 & 0 & 0 & 0\\
            0 & \sigma & 0 & 0 & 0 & 0                    
              \end{bmatrix}.
$$

Using the matrix function $F^{\star}(x)$, define the quasi-derivatives 
\begin{equation} \label{quasis}
z^{[0]} := z, \quad z^{[k]} := (z^{[k-1]})' - \sum_{j = 1}^k f^{\star}_{k,j} z^{[j-1]}, \quad k = \overline{1,n},
\end{equation}
the domain
\begin{equation} \label{defDFs}
\mathcal D_{F^{\star}} := \{ z \colon z^{[k]} \in AC[0,1], \, k = \overline{0,n-1} \},
\end{equation}
and the differential expression 
\begin{equation} \label{deflns}
\ell_n^{\star}(z) := (-1)^n z^{[n]}.
\end{equation}
Note that, in \eqref{defDFs} and \eqref{deflns}, we use the quasi-derivatives defined by \eqref{quasis}. Below we call a function $z$ \textit{a solution} of the differential equation 
\begin{equation} \label{eqvs}
\ell_n^{\star}(z) = \la z, \quad x \in (0,1),
\end{equation}
if $z \in \mathcal D_{F^{\star}}$ and the equality \eqref{eqvs} holds a.e. on $(0,1)$. Throughout this paper, we always use the quasi-derivatives \eqref{quasi} for functions of $\mathcal D_F$ and the quasi-derivatives \eqref{quasis} for functions of $\mathcal D_{F^{\star}}$.

For $z \in \mathcal D_{F^{\star}}$, relations \eqref{entfs} and \eqref{quasis} imply
$$
z^{[k]} = (z^{[k-1]})' + (-1)^k p_{n-k} z, \quad k = \overline{2,n-2}.
$$
Therefore, one can show by induction that
\begin{align} \label{zk}
z^{[k]} & = z^{(k)} + \sum_{j = 0}^{k-2} \left( \sum_{s = j}^{k-2}(-1)^{s+k} C_s^j p_{n-k+s}^{(s-j)}\right) z^{(j)}, \quad k = \overline{0,n-2}, \\ \label{zn1}
z^{[n-1]} & = z^{(n-1)} + \sum_{j = 0}^{n-3} \left( \sum_{s = j}^{n-3} (-1)^{s+n-1} C_s^j p_{s+1}^{(s - j)}\right) z^{(j)} + (-1)^n \sigma z, \\ \nonumber
z^{[n]} & = (z^{[n-1]})' + (-1)^{n+1} \sigma z'.
\end{align}

Consequently, induction implies $z^{(k)} \in AC[0,1]$ for $k = \overline{0,n-2}$. Hence $\mathcal D_{F^{\star}} \subset W_1^{n-1}[0,1]$.

For $z \in \mathcal D_{F^{\star}}$ and $y \in \mathcal D_F$,
define the Lagrange bracket:
\begin{equation} \label{defLagr}
\langle z, y \rangle = \sum_{k = 0}^{n-1} (-1)^k z^{[k]} y^{[n-k-1]}.
\end{equation}

Then, the Lagrange identity holds (see \cite[Section 2]{Bond22-alg}):
$$
\frac{d}{dx} \langle z, y \rangle = z \ell_n(y) - y \ell_n^{\star}(z).
$$

In particular, if $z$ and $y$ solve the equations $\ell^{\star}(z) = \mu z$ and $\ell(y) = \la y$, respectively, then we get
\begin{equation} \label{wron}
\frac{d}{dx} \langle z, y \rangle = (\la - \mu) yz.
\end{equation}

Substituting \eqref{zk} and \eqref{zn1} into \eqref{defLagr}, we derive the relation
\begin{equation} \label{reprLagr}
\langle z, y \rangle = \sum_{k = 0}^{n-1} (-1)^{n-k-1} z^{(n-k-1)} y^{(k)} + \sum_{k = 0}^{n-3} y^{(k)} \sum_{j = 0}^{n-k-3} \left( \sum_{s = j}^{n-k-3} (-1)^s C_s^j p_{s+k+1}^{(s-j)}\right) z^{(j)},
\end{equation}
where all the derivatives are regular, since $z, y \in W_1^{n-1}[0,1]$ and $p_k \in W_2^{k-1}[0,1]$, $k = \overline{1,n-2}$.

\begin{remark} \label{rem:2}
The associated matrix $F(x)$ given by \eqref{entf} regularizes the differential expression $\ell_n(y)$ only for $n \ge 3$. For $n = 2$, the associated matrix is constructed in a different way (see \cite{SS03}):
$$
F = \begin{bmatrix}
        -\sigma & 0 \\
        -\sigma^2 & \sigma
    \end{bmatrix}.
$$

Nevertheless, the main result of this paper (Theorem~\ref{thm:loc}) holds for $n = 2$ and, moreover, has been already proved in \cite{HM-sd} for real-valued potentials.
Therefore, below in the proofs, we confine ourselves to the case $n \ge 3$ and use the associated matrix \eqref{entf}. For $n = 2$, the proofs are valid with minor modifications.
\end{remark}

\begin{remark} \label{rem:families}
For regularization of the differential expression $\ell_n(y)$, different associated matrices can be used (see \cite{MS16, Bond23-mmas, Bond23-reg}). However, it has been proved in \cite{Bond23-reg} that the spectral data $\{ \la_{l,k}, \be_{l,k} \}_{l \ge 1, \, k = \overline{1,n-1}}$ do not depend on the choice of the associated matrix.
\end{remark}

\section{Spectral data} \label{sec:sd}

In this section, we discuss the properties of the spectral characteristics for the boundary value problems $\mathcal L_k$, $k = \overline{1,n-1}$. In particular, the weight numbers $\{ \be_{l,k} \}$ are defined as the residues of some entries of the Weyl-Yurko matrix.

For $k = \overline{1, n}$, denote by $\mathcal C_k(x, \la)$ and $\Phi_k(x, \la)$ the solutions of equation \eqref{eqv} satisfying the initial conditions
$$
\mathcal C_k^{[j-1]}(0, \la) = \de_{k,j}, \quad j = \overline{1,n},
$$
and the boundary conditions
$$
\Phi_k^{[j-1]}(0, \la) = \de_{k,j}, \quad j = \overline{1,k}, \qquad
\Phi_k^{[s-1]}(1, \la) = 0, \quad s = \overline{1,n-k}.
$$
respectively. The functions $\{ \Phi_k(x, \la) \}_{k = 1}^n$ are called \textit{the Weyl solutions} of equation \eqref{eqv}.

Let us summarize the properties of the solutions $\mathcal C_k(x, \la)$ and $\Phi_k(x, \la)$. For details, see \cite{Bond21}.
The functions $\mathcal C_k(x, \la)$, $k = \overline{1,n}$, are uniquely defined as solutions of the initial value problems, and they are entire in $\la$ for each fixed $x \in [0,1]$ together with their quasi-derivatives $\mathcal C_k^{[j]}(x, \la)$, $j = \overline{1,n-1}$. The Weyl solutions $\Phi_k(x, \la)$, $k = \overline{1,n}$, and their quasi-derivatives are meromorphic in $\la$. Furthermore, the fundamental matrices
$$
\mathcal C(x, \la) := [\mathcal C_k^{[j-1]}(x, \la)]_{j,k = 1}^n, \quad \Phi(x, \la) := [\Phi_k^{[j-1]}(x, \la)]_{j,k = 1}^n
$$
are related to each other as follows:
$$
\Phi(x, \la) = \mathcal C(x, \la) M(\la),
$$
where $M(\la) = [M_{j,k}(\la)]_{j,k = 1}^n$ is called \textit{the Weyl-Yurko matrix}. 

The entries $M_{j,k}(\la)$ satisfy the relations
\begin{align} \label{propM1}
M_{j,k}(\la) & = \de_{j,k}, \quad j \le k, \\ \label{propM2}
M_{j,k}(\la) & = -\frac{\Delta_{j,k}(\la)}{\Delta_{k,k}(\la)}, \quad k = \overline{1,n-1}, \quad j = \overline{k+1, n},
\end{align}
where $\Delta_{k,k}(\la) := \det([\mathcal C_j^{[n-s]}(1,\la)]_{s,j = k+1}^n)$, $k = \overline{1,n-1}$, and $\Delta_{j,k}(\la)$ is obtained from $\Delta_{k,k}(\la)$ by replacing $\mathcal C_j$ by $\mathcal C_k$. Clearly, the functions $\Delta_{j,k}(\la)$ for $j \ge k$ are entire in $\la$.
Thus, $M(\la)$ is a unit lower-triangular matrix, whose entries under the main diagonal are meromorphic in $\la$, and the poles of the $k$-th column coincide with the zeros of $\Delta_{k,k}(\la)$.

On the other hand, the zeros of $\Delta_{k,k}(\la)$ coincide with the eigenvalues of the problem $\mathcal L_k$ for equation \eqref{eqv} with the boundary conditions \eqref{bc} for each $k = \overline{1,n-1}$. 
Therefore, under the assumption (W-2) of Definition~\ref{def:W}, all the poles of $M(\la)$ are simple. The Laurent series at $\la = \la_{l,k}$ has the form
$$
M(\la) = \frac{M_{\langle -1 \rangle}(\la_{l,k})}{\la - \la_{l,k}} + 
M_{\langle 0 \rangle}(\la_{l,k}) + M_{\langle 1 \rangle}(\la_{l,k}) (\la - \la_{l,k}) + \dots,
$$
where $M_{\langle j \rangle}(\la_{l,k})$ are the corresponding $(n \times n)$-matrix coefficients.

Define \textit{the weight matrices}
$$
\mathcal N(\la_{l,k}) := (M_{\langle 0 \rangle}(\la_{l,k}))^{-1} M_{\langle -1 \rangle}(\la_{l,k}).
$$

Theorem~4.4 of \cite{Bond23-reg} implies the following uniqueness result.

\begin{prop}[\cite{Bond23-reg}] \label{prop:uniq}
The spectral data $\{ \la_{l,k}, \mathcal N(\la_{l,k}) \}_{l \ge 1, \, k = \overline{1,n-1}}$ uniquely specify the coefficients $p = (p_k)_{k = 0}^{n-2}$ satisfying the conditions (W-1) and (W-2) of Definition~\ref{def:W}.
\end{prop}

The structural properties of the weight matrices $\mathcal N(\la_{l,k}) = [\mathcal N_{s,j}(\la_{l,k})]_{s,j = 1}^n$ are similar to the ones for the case of regular coefficients (see \cite{Yur02, Bond22-alg}).
In view of \eqref{propM1}, $\mathcal N_{s,j}(\la_{l,k}) = 0$ for $s \le j$.
Moreover, under the condition (W-3), $\mathcal N_{s,j}(\la_{l,k}) = 0$ for $s > j+1$. Thus, the only non-zero entries of $\mathcal N(\la_{l,k})$ are $\mathcal N_{j+1,j}(\la_{l,k})$ for such $j$ that $\Delta_{j,j}(\la_{l,k}) = 0$. Therefore, instead of the weight matrices $\mathcal N(\la_{l,k})$, it is sufficient to use \textit{the weight numbers}
$$
\be_{l,k} := \mathcal N_{k+1,k}(\la_{l,k}) = -\frac{\Delta_{k+1,k}(\la_{l,k})}{\frac{d}{d\la} \Delta_{k,k}(\la_{l,k})}. 
$$

Indeed, if $\la_{l_1, k_1} = \la_{l_2,k_2} = \dots = \la_{l_r,k_r}$ is a group of equal eigenvalues (of different problems $\mathcal L_{k_j}$), that is maximal by inclusion, we have
$$
\mathcal N(\la_{l_1, k_1}) = \mathcal N(\la_{l_2, k_2}) = \dots = \mathcal N(\la_{l_r, k_r}) = \sum_{s = 1}^r \be_{l_s, k_s} E_{k_s+1, k_s},
$$
where $E_{i,j}$ denotes the matrix with the unit entry at the position $(i,j)$ and all the other entries equal zero. Hence, Proposition~\ref{prop:uniq} implies the following corollary.

\begin{cor} \label{cor:uniq}
The spectral data $\{ \la_{l,k}, \be_{l,k} \}_{l \ge 1, \, k = \overline{1,n-1}}$ uniquely specify the coefficients $p = (p_k)_{k = 0}^{n-2} \in W$.
\end{cor}

By Theorem~6.2 of \cite{Bond22-asympt}, the weight numbers have the asymptotics
\begin{equation} \label{asymptbe}
\be_{l,k} = l^n (\be_k + \varkappa_{l,k}^0), \quad l \ge 1, \, k = \overline{1,n-1},
\end{equation}
where $\be_k \ne 0$ and $\{ \varkappa_{l,k}^0 \} \in l_2$.

\begin{example} \label{ex:2}
For $n = 2$, $\{ \la_{l,1} \}_{l \ge 1}$ are the Dirichlet eigenvalues of the Sturm-Liouville equation $y'' + p_0 y = \la y$ and
$$
\mathcal N(\la_{l,1}) = \begin{bmatrix}
                        0 & 0 \\
                        \be_{l,1} & 0
                    \end{bmatrix}.
$$

One can easily check that $\be_{l,1} = \alpha_l^{-1}$, where $\alpha_l := \int_0^1 y_l^2(x) \, dx$ and $y_l(x)$ is the eigenfunction of $\la_{l,1}$ such that $y_l^{[1]}(0) = 1$ for a distributional potential $p_0$ or $y'(0) = 1$ for an integrable potential $p_0$.
Therefore, $\{ \la_{l,1}, \be_{l,1} \}_{l \ge 1}$ are equivalent to the spectral data $\{ \la_{l,1}, \al_l \}_{l \ge 1}$ of the classical inverse Sturm-Liouville problem studied by Marchenko \cite{Mar11}, Gelfand and Levitan \cite{GL51}, etc.
\end{example}

\begin{example}
For $n = 3$, (W-3) implies $\{ \la_{l,1} \} \cap \{ \la_{l,2}\} = \varnothing$. Hence
$$
\mathcal N(\la_{l,1}) = \begin{bmatrix}
                            0 & 0 & 0 \\
                            \be_{l,1} & 0 & 0 \\
                            0 & 0 & 0 
                        \end{bmatrix},
\quad
\mathcal N(\la_{l,2}) = \begin{bmatrix}
                            0 & 0 & 0 \\
                            0 & 0 & 0 \\
                            0 & \be_{l,2} & 0
                        \end{bmatrix}.
$$

The recovery of the coefficients $p = (p_0, p_1)\in W$ from the spectral data $\{ \la_{l,k}, \be_{l,k} \}_{l\ge 1, \, k = 1,2}$ has been investigated in \cite{Bond23-res}.
\end{example}

\section{Main equation} \label{sec:main}

In this section, we reduce Problem~\ref{ip:main} to a linear equation in the Banach space $m$ of bounded infinite sequences. First, we deduce an infinite system of linear equations. Second, this system is transformed to achieve the absolute convergence of the series by the method of \cite{Yur02}. Although we rely on the general approach of \cite{Bond22-alg}, the construction of the main equation is simplified because of the separation condition (W-3).

Consider the two coefficient vectors $p = (p_k)_{k = 0}^{n-2}$ and $\tilde p = (\tilde p_k)_{k = 0}^{n-2}$ of the class $W$.  Note that the differential expression $\tilde \ell_n(y)$ with the coefficients $\tilde p$ has the associated matrix $\tilde F(x)$, which can be different from $F(x)$, so the corresponding quasi-derivatives differ. The matrix $\tilde F^{\star}(x)$ and the corresponding quasi-derivatives are defined analogously to $F^{\star}(x)$ and \eqref{quasis}, respectively.

For $k = \overline{1,n}$, denote by $\tilde \Phi^{\star}_k(x, \la)$ the solution of the boundary value problem
\begin{gather*}
\tilde \ell_n^{\star}(\tilde \Phi_k^{\star}) = \la \tilde \Phi_k^{\star}, \quad x \in (0,1), \\
\tilde \Phi_k^{\star[j-1]}(0,\la) = \de_{k,j}, \quad j = \overline{1,k}, \qquad \tilde \Phi_k^{\star[s-1]}(1, \la) = 0, \quad s = \overline{1,n-k}.
\end{gather*}

Introduce the notations
\begin{gather} \nonumber 
    V := \{ (l,k,\eps) \colon l \in \mathbb N, \, k = \overline{1,n-1}. \, \eps = 0, 1 \},  \\ \nonumber
    \la_{l,k,0} := \la_{l,k}, \quad \la_{l,k,1} := \tilde \la_{l,k}, \quad \be_{l,k,0} := \be_{l,k}, \quad \be_{l,k,1} := \tilde \be_{l,k}, \\ \label{defvv}
    \vv_{l,k,\eps}(x) := \Phi_{k+1}(x, \la_{l,k,\eps}), \quad \tilde \vv_{l,k,\eps}(x) := \tilde \Phi_{k+1}(x, \la_{l,k,\eps}), \quad (l,k,\eps) \in V, \\ \label{defG}
    \tilde G_{(l,k,\eps), (l_0, k_0, \eps_0)}(x) :=  \begin{cases}
        (-1)^{n-k} \be_{l,k,\eps} \int_0^x \tilde \Phi_{n-k+1}^{\star}(t, \la_{l,k,\eps}) \tilde \Phi_{k_0 + 1}(t, \la_{l_0,k_0,\eps_0}) \, dt, & (l,k,\eps) = (l_0,k_0,\eps_0), \\ 
        (-1)^{n-k} \be_{l,k,\eps} \dfrac{\langle \tilde \Phi^{\star}_{n-k+1}(x, \la_{l,k,\eps}), \tilde \Phi_{k_0+1}(x, \la_{l_0,k_0,\eps_0})\rangle}{\la_{l_0,k_0,\eps_0} - \la_{l,k,\eps}}, & (l,k,\eps) \ne (l_0,k_0,\eps_0).
    \end{cases}
\end{gather}

Note that, for each fixed $k \in \{ 1, 2, \dots, n-2 \}$, the Weyl solution $\Phi_{k+1}(x, \la)$ has the poles $\{ \la_{l,k+1,0} \}_{l \ge 1}$, which do not coincide with $\{ \la_{l,k,0} \}_{l \ge 1}$ because of (W-3). Furthermore, the solution $\Phi_n(x, \la) \equiv \mathcal C_n(x, \la)$ is entire in $\la$.
For technical convenience, assume that $\{ \la_{l,k} \}_{l \ge 1, \, k = \overline{1,n-1}} \cap \{ \tilde \la_{l,k} \}_{l \ge 1, \, k = \overline{1,n-1}} = \varnothing$. The opposite case requires minor changes. Hence, the functions $\vv_{l,k,\eps}(x)$ are correctly defined by \eqref{defvv}, and so do $\tilde \vv_{l,k,\eps}(x)$, $(l,k,\eps) \in V$.

It has been shown in \cite{Bond22-alg} that $\tilde \Phi^{\star}_j(x, \la)$ has the poles $\tilde \la_{l,j}^{\star} = \tilde \la_{l,n-j-1}$, $l \ge 1$, for $j = \overline{1,n-1}$, and $\tilde \Phi^{\star}_n(x, \la)$ is entire in $\la$. Therefore, $\tilde \Phi_{n-k+1}^{\star}(x, \la_{l,k,\eps})$ is correctly defined and so do $\tilde G_{(l,k,\eps), (l_0,k_0,\eps_0)}(x)$ for $(l,k,\eps), (l_0,k_0,\eps_0) \in V$.

In \cite{Bond22-alg}, the following infinite linear system has been obtained:
\begin{equation} \label{infphi}
\vv_{l_0, k_0,\eps_0}(x) = \tilde \vv_{l_0,k_0,\eps_0}(x) + \sum_{(l,k,\eps) \in V}(-1)^{\eps} \vv_{l,k,\eps}(x) \tilde G_{(l,k,\eps), (l_0, k_0,\eps_0)}(x), \quad (l_0,k_0,\eps_0) \in V.
\end{equation}

Our next goal is to combine the terms in \eqref{infphi} for achieving the absolute convergence of the series.
Introduce the numbers
\begin{equation} \label{defxi}
\xi_l  := \sum_{k = 1}^{n-1} \left( l^{-(n-1)} |\la_{l,k} - \tilde \la_{l,k}| + l^{-n} |\be_{l,k} - \tilde \be_{l,k}| \right), \quad l \ge 1, 
\end{equation}
which characterize the ``distance'' between the spectral data $\{ \la_{l,k}, \be_{l,k} \}_{l \ge 1, \, k = \overline{1,n-1}}$ and $\{ \tilde \la_{l,k}, \tilde \be_{l,k} \}_{l \ge 1, \, k = \overline{1,n-1}}$ of $p$ and $\tilde p$, respectively. The asymptotics \eqref{asymptla} and \eqref{asymptbe} imply that $\{ \xi_l \} \in l_2$. In addition, define the functions
$$
w_{l,k}(x) := l^{-k} \exp(-xl \cot(k\pi/n)),
$$
which characterize the growth of $\varphi_{l,k,\eps}(x)$: the estimate $|\varphi_{l,k,\eps}(x)| \le C w_{l,k}(x)$ holds by Lemma~7 in \cite{Bond22-alg}.

Pass to the new variables
\begin{equation} \label{defpsi}
\begin{bmatrix}
\psi_{l,k,0}(x) \\ \psi_{l,k,1}(x)
\end{bmatrix} := 
w_{l,k}^{-1}(x)
\begin{bmatrix}
\xi_l^{-1} & -\xi_l^{-1} \\ 0 & 1
\end{bmatrix}
\begin{bmatrix}
\vv_{l,k,0}(x) \\ \vv_{l,k,1}(x)
\end{bmatrix}, 
\end{equation}
\begin{multline} \label{defR}
\begin{bmatrix}
\tilde R_{(l_0,k_0,0),(l,k,0)}(x) & \tilde R_{(l_0,k_0,0),(l,k,1)}(x) \\
\tilde R_{(l_0,k_0,1),(l,k,0)}(x) & \tilde R_{(l_0,k_0,1),(l,k,1)}(x)
\end{bmatrix} := \\ 
\frac{w_{l,k}(x)}{w_{l_0,k_0}(x)}
\begin{bmatrix}
\xi_{l_0}^{-1} & -\xi_{l_0}^{-1} \\ 0 & 1
\end{bmatrix}
\begin{bmatrix}
\tilde G_{(l,k,0),(l_0,k_0,0)}(x) & \tilde G_{(l,k,1),(l_0,k_0,0)}(x) \\
\tilde G_{(l,k,0),(l_0,k_0,1)}(x) & \tilde G_{(l,k,1),(l_0,k_0,1)}(x)
\end{bmatrix}
\begin{bmatrix}
\xi_l & 1 \\ 0 & -1
\end{bmatrix}.
\end{multline}
Analogously to $\psi_{l,k,\eps}(x)$, define $\tilde \psi_{l,k,\eps}(x)$. 
For brevity, denote $v = (l,k,\eps)$, $v_0 = (l_0,k_0,\eps_0)$, $v,v_0 \in V$.
Then, relation \eqref{infphi} is transformed into
\begin{equation} \label{syspsi}
\psi_{v_0}(x) = \tilde \psi_{v_0}(x) + \sum_{v \in V} \tilde R_{v_0,v}(x) \psi_{v}(x), \quad v_0 \in V,
\end{equation}
where
\begin{equation} \label{estpsiR}
|\psi_v(x)|, |\tilde \psi_v(x)| \le C, \quad
|\tilde R_{v_0,v}(x)| \le \frac{C \xi_l}{|l - l_0| + 1}, \quad v_0, v \in V.
\end{equation}

Consider the Banach space $m$ of bounded infinite sequences $a = [a_v]_{v \in V}$ with the norm $\| a \|_m := \sup_{v \in V} |a_v|$. For each fixed $x \in [0,1]$, define the linear operator $\tilde R(x) \colon m \to m$ as follows:
$$
(\tilde R(x) a)_{v_0} = \sum_{v \in V} \tilde R_{v_0,v}(x) a_v, \quad v_0 \in V.
$$

In view of the estimates \eqref{estpsiR}, $\psi(x), \tilde \psi(x) \in m$ and the operator $\tilde R(x)$ is bounded in $m$ for each fixed $x \in [0,1]$. Denote by $I$ the identity operator in $m$. Then, the system \eqref{syspsi} can be represented as a linear equation in the Banach space $m$:
\begin{equation} \label{main}
(I - \tilde R(x)) \psi(x) = \tilde \psi(x), \quad x \in [0,1].
\end{equation}

Equation \eqref{main} is called \textit{the main equation} of Problem~\ref{ip:main}. We have derived \eqref{main} under the assumption that $\{ \la_{l,k}, \be_{l,k} \}$ and $\{ \tilde \la_{l,k}, \tilde \be_{l,k} \}$ are the spectral data of the two problems with the coefficients $p = (p_k)_{k = 0}^{n-2}$ and $\tilde p = (\tilde p_k)_{k = 0}^{n-2}$, respectively. Anyway, the main equation \eqref{main} can be used for the reconstruction of $p$ by $\{ \la_{l,k}, \be_{l,k} \}$. Indeed, one can choose an arbitrary $\tilde p \in W$, find $\tilde \psi(x)$ and $\tilde R(x)$ by using $\tilde p$, $\{ \tilde \la_{l,k}, \tilde \be_{l,k} \}$, and $\{ \la_{l,k}, \be_{l,k} \}$, then find $\psi(x)$ by solving the main equation. In order to find $p$ from $\psi(x)$, we need reconstruction formulas, which are obtained in the next section. 

\section{Reconstruction formulas} \label{sec:rec}

In this section, we derive formulas for recovering the coefficients $(p_k)_{k = 0}^{n-2}$ from the solution $\psi(x)$ of the main equation \eqref{main}. For the derivation, we use the special structure of the associated matrices $F(x)$ and $F^{\star}(x)$. The arguments of this section are based on formal calculations with infinite series. The convergence of those series will be rigorously studied in the next section.

Find $\{ \vv_{l,k,\eps}(x) \}_{(l,k,\eps) \in V}$ from \eqref{defpsi}:
\begin{equation} \label{findvv}
    \begin{bmatrix}
        \vv_{l,k,0}(x) \\ \vv_{l,k,1}(x)
    \end{bmatrix} = w_{l,k}(x)
    \begin{bmatrix}
        \xi_l & 1 \\ 0 & 1
    \end{bmatrix}
    \begin{bmatrix}
        \psi_{l,k,0}(x) \\ \psi_{l,k,1}(x)
    \end{bmatrix}.
\end{equation}

Then, we can recover the Weyl solutions for $k_0 = \overline{1,n}$:
\begin{equation} \label{Phik0}
\Phi_{k_0}(x, \la) := \tilde \Phi_{k_0}(x, \la) + \sum_{(l,k,\eps) \in V} (-1)^{\eps + n-k} \beta_{l,k,\eps}\vv_{l,k,\eps}(x) \frac{\langle \tilde \Phi_{n-k+1}^{\star}(x, \la_{l,k,\eps}), \tilde \Phi_{k_0+1}(x, \la)\rangle}{\la - \la_{l,k,\eps}}.
\end{equation}

Formally applying the differential expression $\ell_n(y)$ to the left- and the right-hand sides of \eqref{Phik0}, after some transforms (see \cite[Section~4.1]{Bond22-alg}), we arrive at the relation
\begin{align} \nonumber
\sum_{(l,k,\eps) \in V} (-1)^{\eps} \vv_{l,k,\eps}(x) \langle \tilde \eta_{l,k,\eps}(x), \tilde \Phi_{k_0}(x, \la) \rangle =  & \sum_{s = 0}^{n-2} \hat p_s(x) \tilde \Phi_{k_0}^{(s)}(x, \la) + \sum_{s = 0}^{n-1} t_{n,s}(x) \tilde \Phi_{k_0}^{(s)}(x, \la) \\ \label{rel1} & + \sum_{s = 0}^{n-3} \sum_{r = s+1}^{n-2} p_r(x) t_{r,s}(x) \tilde \Phi_{k_0}^{(s)}(x, \la),
\end{align}
where $\hat p_s := p_s - \tilde p_s$,
\begin{align} \label{defeta}
\tilde \eta_{l,k,\eps}(x) & := (-1)^{n-k} \be_{l,k,\eps} \tilde \Phi^{\star}_{n-k+1}(x, \la_{l,k,\eps}), \\ \label{deft}
t_{r,s}(x) & := \sum_{u = s}^{r-1} C_r^{u+1} C_u^s T_{r-u-1,u-s}(x), \\ \label{defT}
T_{j_1, j_2}(x) & := \sum_{(l,k,\eps) \in V} (-1)^{\eps} \vv_{l,k,\eps}^{(j_1)}(x) \tilde \eta_{l,k,\eps}^{(j_2)}(x).
\end{align}

Note that, for $(l,k,\eps) \in V$, the functions $\tilde \eta_{l,k,\eps}(x)$ are solutions of the equation $\tilde \ell^{\star}(z) = \la_{l,k,\eps} z$, so the quasi-derivatives for them are generated by the matrix function $\tilde F^{\star}(x)$. Thus, relation \eqref{reprLagr} for the Lagrange bracket in \eqref{rel1} implies
\begin{align} \nonumber
\langle \tilde \eta_{l,k,\eps}(x), \tilde \Phi_{k_0}(x, \la) \rangle = & \sum_{s = 0}^{n-1} (-1)^{n-s-1} \tilde \eta_{l,k,\eps}^{(n-s-1)}(x) \tilde \Phi_{k_0}^{(s)}(x, \la) \\
& \label{rel2} + \sum_{s = 0}^{n-3} \left( \sum_{j = 0}^{n-s-3} \sum_{r = j}^{n-s-3} (-1)^r C_r^j \tilde p_{s+r+1}^{(r-j)}(x) \tilde \eta_{l,k,\eps}^{(j)}(x) \right) \tilde \Phi_{k_0}^{(s)}(x, \la).
\end{align}

Substituting \eqref{rel2} into \eqref{rel1} and using \eqref{defT}, we obtain
\begin{align} \nonumber
& \sum_{s = 0}^{n-1} (-1)^{n-s-1} T_{0,n-s-1}(x) \tilde \Phi_{k_0}^{(s)}(x, \la) + 
\sum_{s = 0}^{n-3} \left( \sum_{j = 0}^{n-s-3} \sum_{r = j}^{n-s-3} (-1)^r C_r^j \tilde p_{s+r+1}^{(r-j)}(x) T_{0,j}(x) \right) \tilde \Phi_{k_0}^{(s)}(x, \la) \\ & = 
\sum_{s = 0}^{n-2} \hat p_s(x) \tilde \Phi_{k_0}^{(s)}(x, \la) + \sum_{s = 0}^{n-1} t_{n,s}(x) \tilde \Phi_{k_0}^{(s)}(x, \la) + \sum_{s = 0}^{n-3} \sum_{r = s+1}^{n-2} p_r(x) t_{r,s}(x) \tilde \Phi_{k_0}^{(s)}(x, \la).
\end{align}

Let us group the terms at $\Phi_{k_0}^{(s)}(x, \la)$ and assume that the corresponding left- and right-hand sides are equal to each other:
\begin{align*}
\Phi_{k_0}^{(n-1)}(x, \la) \colon \quad & T_{0,0}(x) = T_{0,0}(x), \\
\Phi_{k_0}^{(n-2)}(x, \la) \colon \quad & -T_{0,1}(x) = \hat p_{n-2}(x) + t_{n,n-2}(x), \\
\Phi_{k_0}^{(s)}(x, \la) \colon \quad & (-1)^{n-s-1} T_{0,n-s-1}(x) + \sum_{j = 0}^{n-s-3} \sum_{r = j}^{n-s-3} (-1)^r C_r^j \tilde p_{s+r+1}^{(r-j)}(x) T_{0,j}(x) \\
& = \, \hat p_s(x) + t_{n,s}(x) + \sum_{r = s+1}^{n-2} p_r(x) t_{r,s}(x), \quad s = \overline{0, n-3}.
\end{align*}

From here, we obtain the reconstruction formulas
\begin{align} \nonumber
p_s(x) = & \, \tilde p_s(x) - \left( t_{n,s}(x) + (-1)^{n-s} T_{0,n-s-1}(x) \right) \\ \label{recp} & + 
\sum_{j = 0}^{n-s-3} \sum_{r = j}^{n-s-3} (-1)^r C_r^j \tilde p_{r+s+1}^{(r-j)}(x) T_{0,j}(x) - \sum_{r = s+1}^{n-2} p_r(x) t_{r,s}(x), 
\end{align}
for $s = n-2, n-3, \dots, 1, 0$.

Thus, we arrive at the following constructive procedure for solving Problem~\ref{ip:main}.

\begin{alg} \label{alg:1}
Suppose that the spectral data $\{ \la_{l,k}, \be_{l,k} \}_{l \ge 1, \, k = \overline{1,n-1}}$ are given. We have to find the coefficients $p = (p_k)_{k = 0}^{n-2}$.

\begin{enumerate}
\item Choose $\tilde p = (\tilde p_k)_{k = 0}^{n-2} \in W$ and find the spectral data $\{ \tilde \la_{l,k}, \tilde \be_{l,k} \}_{l \ge 1, \, k = \overline{1,n-1}}$. 
\item Find the Weyl solutions $\tilde \vv_{l,k,\eps}(x) = \tilde \Phi_{k+1}(x, \la_{l,k,\eps})$ and $\tilde \Phi_{n-k+1}^{\star}(x, \la_{l,k,\eps})$ for $(l,k,\eps) \in V$, and construct $\tilde G_{(l,k,\eps),(l_0,k_0,\eps_0)}(x)$ by \eqref{defG}.
\item Construct $\tilde \psi_v(x)$ for $v \in V$ by \eqref{defpsi} and $\tilde R_{v_0,v}(x)$ for $v_0, v \in V$ by \eqref{defR}.
\item Find $\psi_v(x)$, $v \in V$, by solving the main equation \eqref{main}.
\item For $(l,k,\eps) \in V$, determine $\vv_{l,k,\eps}(x)$ by \eqref{findvv} and $\tilde \eta_{l,k,\eps}(x)$ by \eqref{defeta}.
\item For $s = n-2, n-3, \dots, 1, 0$, find $p_s(x)$ by formula \eqref{recp}, in which $t_{r,s}(x)$ and $T_{j_1,j_2}(x)$ are defined by \eqref{deft} and \eqref{defT}, respectively.
\end{enumerate}
\end{alg}

Procedure~\ref{alg:1} will be used in the next section for proving Theorem~\ref{thm:loc}. In general, there is a challenge to choose a model problem $\tilde p$ so that the series for $p_s$ converge in the corresponding spaces. Note that steps~1--5 work for any $\tilde p \in W$, since for them the estimate $\{ \xi_l \} \in l_2$ is sufficient. But the situation differs for step~6. In Section~\ref{sec:solve}, we prove the validity of step~6 in the case $\{ l^{n-2} \xi_l \} \in l_2$.

\section{Solvability and stability} \label{sec:solve}

In this section, we prove the following theorem on the solvability of Problem~\ref{ip:main}.

\begin{thm} \label{thm:glob}
Let complex numbers $\{ \la_{l,k}, \be_{l,k} \}_{l \ge 1, \, k = \overline{1,n-1}}$ and coefficients $\tilde p = (\tilde p_k)_{k = 0}^{n-2} \in W$ be such that:

\smallskip

\textbf{(S-1)} For each $k = \overline{1,n-1}$, the numbers $\{ \la_{l,k} \}_{l \ge 1}$ are distinct.

\smallskip

\textbf{(S-2)} $\{ \la_{l,k} \}_{l \ge 1} \cap \{ \la_{l,k+1} \}_{l \ge 1} = \varnothing$ for $k = \overline{1,n-2}$.

\smallskip

\textbf{(S-3)} $\be_{l,k} \ne 0$ for all $l \ge 1$, $k = \overline{1,n-1}$.

\smallskip

\textbf{(S-4)} $\{ l^{n-2} \xi_l \}_{l \ge 1} \in l_2$, where the numbers $\xi_l$ are defined in \eqref{defxi}.

\smallskip

\textbf{(S-5)} The operator $(I - \tilde R(x))$, which is constructed by using $\{ \la_{l,k}, \be_{l,k} \}$ and $\tilde p$ according to Section~\ref{sec:main}, has a bounded inverse operator for each fixed $x \in [0,1]$.

\smallskip

Then, $\{ \la_{l,k}, \be_{l,k} \}_{l \ge 1, \, k = \overline{1,n-1}}$ are the spectral data of some (unique) $p = (p_k)_{k = 0}^{n-2} \in W$.
\end{thm}

Theorem~\ref{thm:glob} provides sufficient conditions for global solvability of the inverse problem. Theorem~\ref{thm:loc} on local solvability and stability will be obtained as a corollary of Theorem~\ref{thm:glob}. Thus, Theorem~\ref{thm:glob} plays an auxiliary role in this paper but also has a separate significance. The proof of Theorem~\ref{thm:glob} is based on Procedure~\ref{alg:1}. We investigate the properties of the solution $\psi(x)$ of the main equation and prove the convergence of the series in \eqref{defT} and \eqref{recp} in the corresponding spaces of regular and generalized functions. This part of the proofs is the most difficult one, since the series converge in different spaces and precise estimates for the Weyl solutions are needed. Finally, we show that the numbers $\{ \la_{l,k}, \be_{l,k} \}$ satisfying the conditions of Theorem~\ref{thm:glob} are the spectral data of the coefficients $p = (p_k)_{k = 0}^{n-2}$ reconstructed by formulas \eqref{recp}. In the end of this section, we prove Theorem~\ref{thm:loc}.

Proceed to the proof of Theorem~\ref{thm:glob}.
Let $\{ \la_{l,k}, \be_{l,k} \}_{l \ge 1, \, k = \overline{1,n-1}}$ and $\tilde p$ satisfy the hypotheses (S-1)--(S-5). We emphasize that $\{ \la_{l,k}, \be_{l,k} \}$ are not necessarily the spectral data corresponding to some $p$. We have to prove this. 

By virtue of (S-5), the operator $(I - \tilde R(x))$ has a bounded inverse. Therefore, the main equation \eqref{main} is uniquely solvable in $m$ for each fixed $x \in [0,1]$. Consider its solution $\psi(x) = [\psi_v(x)]_{v \in V}$. Recover the functions $\vv_{l,k,\eps}(x)$ for $(l,k,\eps) \in V$ by \eqref{findvv}. Let us study their properties. For this purpose, we need the auxiliary estimates for $\tilde \Phi_k(x, \la)$ and $\tilde \Phi^{\star}_k(x, \la)$, which were deduced from the results of \cite{Bond21} and used in Section~4.2 of \cite{Bond22-alg}.

\begin{prop} [\cite{Bond21, Bond22-alg}] \label{prop:Phi}
For $(l,k,\eps) \in V$, $x \in [0,1]$, and $\nu = \overline{0,n-1}$, the following estimates hold:
\begin{gather*}
|\tilde \Phi_{k+1}^{[\nu]}(x, \la_{l,k,\eps})| \le C l^{\nu} w_{l,k}(x), \quad
|\tilde \Phi_{k+1}^{[\nu]}(x, \la_{l,k,0}) - \tilde \Phi_{k+1}^{[\nu]}(x, \la_{l,k,1})| \le C l^{\nu} w_{l,k}(x) \xi_l, \\
|\tilde \Phi_{n-k+1}^{\star [\nu]}(x, \la_{l,k,\eps})| \le C l^{\nu - n} w_{l,k}^{-1}(x), \quad
|\tilde \Phi_{n-k+1}^{\star [\nu]}(x, \la_{l,k,0}) - \tilde \Phi_{n-k+1}^{\star [\nu]}(x, \la_{l,k,1})| \le C l^{\nu - n} w_{l,k}^{-1}(x) \xi_l, 
\end{gather*}
where $w_{l,k}(x) := l^{-k} \exp(-xl \cot (k\pi/n))$.
\end{prop}

\begin{cor} \label{cor:Phi}
The functions $\tilde \Phi_k(., \la)$ and $\tilde \Phi_k^{\star}(.,\la)$ belong to $C^{n-2}[0,1]$ for $k = \overline{1,n}$.
Moreover, for $\nu = \overline{0,n-2}$, the estimates of Proposition~\ref{prop:Phi} are valid for the quasi-derivatives replaced with the classical derivatives:
\begin{gather} \label{estPhi}
|\tilde \Phi_{k+1}^{(\nu)}(x, \la_{l,k,\eps})| \le C l^{\nu} w_{l,k}(x), \quad
|\tilde \Phi_{k+1}^{(\nu)}(x, \la_{l,k,0}) - \tilde \Phi_{k+1}^{(\nu)}(x, \la_{l,k,1})| \le C l^{\nu} w_{l,k}(x) \xi_l, \\ \label{estPhis}
|\tilde \Phi_{n-k+1}^{\star (\nu)}(x, \la_{l,k,\eps})| \le C l^{\nu - n} w_{l,k}^{-1}(x), \quad
|\tilde \Phi_{n-k+1}^{\star (\nu)}(x, \la_{l,k,0}) - \tilde \Phi_{n-k+1}^{\star (\nu)}(x, \la_{l,k,1})| \le C l^{\nu - n} w_{l,k}^{-1}(x) \xi_l, 
\end{gather}
where $(l,k,\eps) \in V$, $x \in [0,1]$.
\end{cor}

\begin{proof}
Recall that $\{ \tilde \Phi_k(x, \la) \}_{k = 1}^n$ and $\{ \tilde \Phi_k^{\star}(x,\la) \}_{k = 1}^n$ are solutions of the equations $\tilde \ell(y) = \la y$ and $\tilde \ell^{\star}(z) = \la z$, respectively. Hence 
$$
\tilde \Phi_k(., \la) \in \mathcal D_{\tilde F} \subset W_1^{n-1}[0,1], \quad
\tilde \Phi_k^{\star}(., \la) \in \mathcal D_{\tilde F^{\star}} \subset W_1^{n-1}[0,1].
$$
This implies $\tilde \Phi_k(.,\la), \tilde \Phi_k^{\star}(.,\la) \in C^{n-2}[0,1]$.

In view of \eqref{quas1}, we have $\tilde \Phi_k^{[\nu]}(x, \la) = \tilde \Phi_k^{(\nu)}(x, \la)$ for $\nu = \overline{0,n-2}$, so the estimates \eqref{estPhi} readily follow from Proposition~\ref{prop:Phi}.

The quasi-derivatives for $\tilde \Phi_{n-k+1}^{\star}(x, \la)$ are generated by the matrix function $\tilde F^{\star}(x)$. Hence, the relation similar to \eqref{zk} is valid for them:
\begin{equation} \label{quasieta}
\tilde \Phi^{\star[\nu]}_{n-k+1}(x,\la) = \tilde \Phi^{\star(\nu)}_{n-k+1}(x,\la) + \sum_{j = 0}^{\nu - 2} \left( \sum_{s = j}^{\nu-2} (-1)^{s + \nu} C_s^j \tilde p_{n-\nu + s}^{(s-j)}(x)\right) \tilde \Phi^{\star(j)}_{n-k+1}(x,\la), 
\end{equation}
where $\nu = \overline{0,n-2}$, $k = \overline{1,n}$.

Since $\tilde p_k \in W_2^{k-1}[0,1]$ for $k = \overline{0,n-2}$, we get that all the derivatives $\tilde p_{n-\nu + s}^{(s-j)}$ in \eqref{quasieta} belong to $W_1^1[0,1]$, so they are bounded. Let us prove the estimates \eqref{estPhis} by induction. For $\nu = 0, 1$, we have $\tilde \Phi_{n-k+1}^{\star(\nu)}(x, \la) = \tilde \Phi_{n-k+1}^{\star[\nu]}(x, \la)$, so the estimates \eqref{estPhis} directly follow from Proposition~\ref{prop:Phi}. Next, consider $\nu \ge 2$ and assume that the estimates \eqref{estPhis} are already proved for $0, 1, \dots, \nu - 1$. Then, using \eqref{quasieta}, Proposition~\ref{prop:Phi}, and the induction hypothesis, we get
$$
|\tilde \Phi^{\star(\nu)}_{n-k+1}(x,\la_{l,k,\eps})| \le |\tilde \Phi^{\star[\nu]}_{n-k+1}(x,\la_{l,k,\eps})| + C \sum_{j = 0}^{\nu - 2} |\tilde \Phi^{\star(j)}_{n-k+1}(x, \la_{l,k,\eps})| \le C l^{\nu - n} w_{l,k}^{-1}(x).
$$

The second estimate in \eqref{estPhis} is obtained similarly.
\end{proof}

Comparing \eqref{condloc} to \eqref{defxi} and taking (S-4) into account, we conclude that
\begin{equation}
\label{defOmega}
\Omega  = \left( \sum_{l = 1}^{\infty} (l^{n-2} \xi_l)^2 \right)^{1/2}< \infty.
\end{equation}

\begin{lem} \label{lem:vv}
For $(l,k,\eps) \in V$, we have $\vv_{l,k,\eps} \in C^{n-2}[0,1]$ and 
\begin{gather*}
|\vv_{l,k,\eps}^{(\nu)}(x)| \le C l^{\nu} w_{l,k}(x), \quad
|\vv_{l,k,0}^{(\nu)}(x) - \vv_{l,k,1}^{(\nu)}(x)| \le C l^{\nu} w_{l,k}(x) \xi_l, \quad \nu = \overline{0,n-2}, \\
|\vv_{l,k,\eps}(x) - \tilde \vv_{l,k,\eps}(x)| \le C \Omega w_{l,k}(x) \chi_l, \\
|\vv_{l,k,0}(x) - \vv_{l,k,1}(x) - \tilde \vv_{l,k,0}(x) + \tilde \vv_{l,k,1}(x)| \le C \Omega w_{l,k}(x) \chi_l \xi_l, \\
\left.\begin{array}{c}
|\vv_{l,k,\eps}^{(\nu)}(x) - \tilde \vv_{l,k,\eps}^{(\nu)}(x)| \le C \Omega l^{\nu - 1}w_{l,k}(x), \\
|\vv_{l,k,0}^{(\nu)}(x) - \vv_{l,k,1}^{(\nu)} - \tilde \vv_{l,k,0}^{(\nu)}(x) + \tilde \vv_{l,k,1}^{(\nu)}(x)| \le C \Omega l^{\nu - 1} w_{l,k}(x) \xi_l, 
\end{array} \right\}
\quad \nu = \overline{1,n-2},
\end{gather*}
where $x \in [0,1]$ and
$$
\chi_l := \left( \sum_{k = 1}^{\infty} \frac{1}{k^2 (|l-k|+1)^2} \right)^{1/2}, \quad \{ \chi_l \}_{l \ge 1} \in l_2.
$$
\end{lem}

\begin{proof}
First, let us investigate the smoothness of the functions $\tilde \psi_v(x)$ and $\tilde R_{v_0,v}(x)$ in the main equation. 
Due to \eqref{defvv}, \eqref{defpsi}, and Corollary~\ref{cor:Phi}, we get $\tilde \vv_{l,k,\eps} \in C^{n-2}[0,1]$ for $(l,k,\eps) \in V$ and so
$\tilde \psi_v \in C^{n-2}[0,1]$ for $v \in V$. Next, applying \eqref{wron} to \eqref{defG}, we obtain
\begin{equation} \label{Gt}
\tilde G'_{(l,k,\eps),(l_0,k_0,\eps_0)}(x) = (-1)^{n-k} \be_{l,k,\eps} \tilde \Phi^{\star}_{n-k+1}(x, \la_{l,k,\eps}) \tilde \Phi_{k_0+1}(x, \la_{l_0,k_0,\eps_0}) \in C^{n-2}[0,1].
\end{equation}

Using \eqref{Gt} together with \eqref{defR}, we conclude that $\tilde G_{(l,k,\eps),(l_0,k_0,\eps_0)} \in C^{n-1}[0,1]$ and so $\tilde R_{v_0,v} \in C^{n-1}[0,1]$. 
Moreover, using \eqref{defvv} and Corollary~\ref{cor:Phi}, we obtain
\begin{equation} \label{estvt}
|\tilde \vv^{(\nu)}_{l,k,\eps}(x)| \le C l^{\nu} w_{l,k}(x), \quad
|\tilde \vv^{(\nu)}_{l,k,0}(x) - \tilde \vv^{(\nu)}_{l,k,1}(x)| \le C l^{\nu} w_{l,k}(x) \xi_l, 
\end{equation}
for $\nu = \overline{0,n-2}$, $(l,k,\eps) \in V$. 

The asymptotics \eqref{asymptbe} and \eqref{defxi} imply the following estimates for the weight numbers:
\begin{equation} \label{estbeta}
|\be_{l,k,\eps}| \le C l^n, \quad |\be_{l,k,0} - \be_{l,k,1}| \le C l^n \xi_l, \quad (l,k,\eps) \in V.
\end{equation}

Using \eqref{Gt}, \eqref{estbeta} and Corollary~\ref{cor:Phi}, we obtain
$$
|\tilde G'_{(l,k,\eps),(l_0,k_0,\eps_0)}(x)| \le C w_{l_0,k_0}(x) w_{l,k}^{-1}(x).
$$

Furthermore,
\begin{multline*}
\tilde G'_{(l,k,0),(l_0,k_0,\eps_0)}(x) - \tilde G'_{(l,k,1),(l_0,k_0,\eps_0)}(x) = 
(-1)^{n-k} (\be_{l,k,0} - \be_{l,k,1}) \tilde \Phi_{n-k+1}^{\star}(x, \la_{l,k,0}) \tilde \Phi_{k_0+1}(x, \la_{l_0,k_0,\eps_0}) \\
+ (-1)^{n-k} \be_{l,k,1} (\tilde \Phi_{n-k+1}^{\star}(x, \la_{l,k,0}) - \tilde \Phi_{n-k+1}^{\star}(x, \la_{l,k,1})) \tilde \Phi_{k_0+1}(x, \la_{l_0,k_0,\eps_0}).
\end{multline*}

Applying the estimates of Corollary~\ref{cor:Phi} and \eqref{estbeta}, we get
$$
|\tilde G'_{(l,k,0),(l_0,k_0,\eps_0)}(x) - \tilde G'_{(l,k,1),(l_0,k_0,\eps_0)}(x)| \le C w_{l_0,k_0}(x) w_{l,k}^{-1}(x) \xi_l. 
$$

Calculating the derivatives of higher-orders for \eqref{Gt}, we similarly obtain 
\begin{gather*}
|\tilde G^{(\nu)}_{(l,k,\eps),(l_0,k_0,\eps_0)}(x)| \le C (l+l_0)^{\nu - 1} w_{l_0,k_0}(x) w_{l,k}^{-1}(x), \\
|\tilde G^{(\nu)}_{(l,k,0),(l_0,k_0,\eps_0)}(x) - \tilde G^{(\nu)}_{(l,k,1),(l_0,k_0,\eps_0)}(x)| \le C (l+l_0)^{\nu - 1} w_{l_0,k_0}(x) w_{l,k}^{-1}(x) \xi_l
\end{gather*}
for $\nu = \overline{2, n-2}$.
Analogously, we deduce
\begin{gather*}
|\tilde G^{(\nu)}_{(l,k,\eps),(l_0,k_0,0)}(x) - \tilde G^{(\nu)}_{(l,k,\eps),(l_0,k_0,1)}(x)| \le C (l+l_0)^{\nu - 1}w_{l_0,k_0}(x) w_{l,k}^{-1}(x) \xi_{l_0}, \\
|\tilde G^{(\nu)}_{(l,k,0),(l_0,k_0,0)}(x) - \tilde G^{(\nu)}_{(l,k,1),(l_0,k_0,0)}(x) - \tilde G^{(\nu)}_{(l,k,0),(l_0,k_0,1)}(x) + \tilde G^{(\nu)}_{(l,k,1),(l_0,k_0,1)}(x)| \\ \le C (l + l_0)^{\nu - 1}w_{l_0,k_0}(x) w_{l,k}^{-1}(x) \xi_{l_0} \xi_l,
\end{gather*}
for $\nu = \overline{1,n-2}$.

Then, using \eqref{defpsi}, \eqref{defR}, \eqref{estvt}, and the above estimates for $\tilde G_{(l,k,\eps),(l_0,k_0,\eps_0)}^{(\nu)}(x)$, we obtain
\begin{equation} \label{estpsiR1}
|\tilde \psi_v^{(\nu)}(x)| \le C l^{\nu}, \quad |\tilde R_{v_0,v}^{(\nu)}(x)| \le C (l + l_0)^{\nu - 1} \xi_l, \quad v_0, v \in V, \quad \nu = \overline{1,n-2}.
\end{equation}
In addition, we have formulas \eqref{estpsiR} for $\tilde \psi_v(x)$, $\tilde R_{v_0,v}(x)$, $\nu = 0$. The obtained estimates coincide with the ones for the case of regular coefficients (see formulas (2.3.40) in \cite{Yur02}). Therefore, the remaining part of the proof almost repeats the proof of Lemma 1.6.7 in \cite{Yur02}, so we omit the technical details. By differentiating the relation $\psi(x) = (I - \tilde R(x))^{-1} \tilde \psi(x)$ and analyzing the convergence of the obtained series, we prove the following properties of $\psi(x) = [\psi_v(x)]_{v \in V}$:
\begin{gather*}
\psi_v \in C^{n-2}[0,1], \quad |\psi_v^{(\nu)}(x)| \le C l^{\nu}, \quad \nu = \overline{0,n-2}, \\
|\psi_v(x) - \tilde \psi_v(x)| \le C \Omega \chi_l, \quad 
|\psi_v^{(\nu)}(x) - \tilde \psi_v^{(\nu)}(x)| \le C \Omega l^{\nu-1}, \quad \nu = \overline{1,n-2}.
\end{gather*}

Using the latter estimates together with \eqref{findvv}, we readily arrive at the claimed estimates for $\vv_{l,k,\eps}(x)$.
\end{proof}

Analogous estimates can be obtained for $\tilde \eta_{l,k,\eps}(x)$ defined by \eqref{defeta}.

\begin{lem} \label{lem:eta}
For $(l,k,\eps) \in V$, we have $\eta_{l,k,\eps} \in C^{n-2}[0,1]$ and
$$
|\tilde \eta_{l,k,\eps}^{(\nu)}(x)| \le C l^{\nu} w_{l,k}^{-1}(x), \quad
|\tilde \eta_{l,k,0}^{(\nu)}(x) - \tilde \eta_{l,k,1}^{(\nu)}(x)| \le C l^{\nu} w_{l,k}^{-1}(x) \xi_l, \quad \nu = \overline{0,n-2}.
$$
\end{lem}

\begin{proof}
The assertion of the lemma immediately follows from the definition \eqref{defeta}, the estimates \eqref{estbeta} for the weight numbers, and Corollary~\ref{cor:Phi}.
\end{proof}

Proceed to the investigation of the convergence for the series $t_{r,s}(x)$ and $T_{j_1, j_2}(x)$ in the reconstruction formulas \eqref{recp}. We rely on Proposition~\ref{prop:series}, which (due to our notations) readily follows from Lemma~8 in \cite{Bond22-alg} and its proof.

\begin{prop}[\cite{Bond22-alg}] \label{prop:series}
The following statements hold.
\begin{enumerate}
\item If $j_1 + j_2 = n-2$, then there exist regularization constants $a_{j_1,j_2,l,k}$ such that the series
$$
\mathscr T_{j_1,j_2}^{reg}(x) := \sum_{l = 1}^{\infty} \sum_{k = 1}^{n-1} \bigl( \tilde \vv_{l,k,0}^{[j_1]}(x) \tilde \eta_{l,k,0}^{[j_2]}(x) - \tilde \vv_{l,k,1}^{[j_1]}(x) \tilde \eta^{[j_2]}_{l,k,1}(x) - a_{j_1,j_2,l,k} \bigr) 
$$
converges in $L_2[0,1]$, $\| \mathscr T_{j_1,j_2}^{reg} \|_{L_2[0,1]} \le C \Omega$, and $a_{j_1,j_2,l,k} + a_{j_1+1,j_2-1,l,k} = 0$, so the series
$$
\sum_{l = 1}^{\infty} \sum_{k = 1}^{n-1} \bigl( \tilde \vv_{l,k,0}^{[j_1]}(x) \tilde \eta_{l,k,0}^{[j_2]}(x) - \tilde \vv_{l,k,1}^{[j_1]}(x) \tilde \eta^{[j_2]}_{l,k,1}(x) + \tilde \vv_{l,k,0}^{[j_1+1]}(x) \tilde \eta_{l,k,0}^{[j_2-1]}(x) - \tilde \vv_{l,k,1}^{[j_1+1]}(x) \tilde \eta^{[j_2-1]}_{l,k,1}(x) \bigr) 
$$
converges in $L_2[0,1]$ without regularization.
\item If $j_1 + j_2 < n-2$, then the series
$$
\mathscr T_{j_1,j_2}(x) := \sum_{l = 1}^{\infty} \sum_{k = 1}^{n-1} \bigl( \tilde \vv_{l,k,0}^{[j_1]}(x) \tilde \eta_{l,k,0}^{[j_2]}(x) - \tilde \vv_{l,k,1}^{[j_1]}(x) \tilde \eta^{[j_2]}_{l,k,1}(x) \bigr) 
$$
converges absolutely and uniformly on $[0,1]$. Moreover, $\max_{x \in [0,1]} |\mathscr T_{j_1,j_2}(x)| \le C \Omega$.
\end{enumerate}
\end{prop}

The constants $a_{j_1, j_2, l,k}$ are explicitly found in the proof of Lemma~8 in \cite{Bond22-alg}. However, we do not provide them here in order not to introduce many additional notations. Moreover, explicit formulas for $a_{j_1, j_2, l,k}$ are not needed in the proofs.

Below, similarly to the series in Proposition~\ref{prop:series}, we consider the series $T_{j_1, j_2}(x)$ with the brackets:
$$
T_{j_1,j_2}(x) = \sum_{l = 1}^{\infty} \sum_{k = 1}^{n-1} \bigl( \vv_{l,k,0}^{(j_1)}(x) \tilde \eta_{l,k,0}^{(j_2)}(x) - \vv_{l,k,1}^{(j_1)}(x) \tilde \eta^{(j_2)}_{l,k,1}(x) \bigr). 
$$

Moreover, we agree that we understand the summation of several series $T_{j_1, j_2}(x)$ (in particular, in \eqref{deft} and in \eqref{recp}) in the sense
\begin{equation} \label{sum}
\sum_{(l,k,\eps) \in V} b_{l,k,\eps} + \sum_{(l,k,\eps) \in V} c_{l,k,\eps} = \sum_{l = 1}^{\infty} \sum_{k = 1}^{n-1} (b_{l,k,0} + b_{l,k,1} + c_{l,k,0} + c_{l,k,1}).
\end{equation}

Using Lemmas~\ref{lem:vv}, \ref{lem:eta} and Proposition~\ref{prop:series}, we get the lemma on the convergence of $T_{j_1, j_2}(x)$.

\begin{lem} \label{lem:Tconv}
The following statements hold.
\begin{enumerate}
\item If $j_1 + j_2 = n-2$, then the series $T_{j_1,j_2}(x)$ converges in $L_2[0,1]$ with the regularization constants $a_{j_1, j_2, l,k}$ from Proposition~\ref{prop:series}, that is, the series
$$
T_{j_1,j_2}^{reg}(x) := \sum_{l = 1}^{\infty} \sum_{k = 1}^{n-1} \bigl( \vv_{l,k,0}^{(j_1)}(x) \tilde \eta_{l,k,0}^{(j_2)}(x) - \vv_{l,k,1}^{(j_1)}(x) \tilde \eta^{(j_2)}_{l,k,1}(x) - a_{j_1,j_2,l,k}\bigr)
$$
converges in $L_2[0,1]$. Moreover, 
\begin{equation} \label{estTreg}
    \| T_{j_1,j_2}^{reg}(x) \|_{L_2[0,1]} \le C \Omega.
\end{equation}
\item If $j_1 + j_2 = n-2-s$ and $s \in \{ 1, 2, \dots, n-2\}$, then $T_{j_1,j_2}(x)$ converges in $W_2^s[0,1]$ and $\| T_{j_1,j_2}(x) \|_{W_2^s[0,1]} \le C \Omega$.
\end{enumerate}
\end{lem}

\begin{proof}
Observe that, for $j_1 + j_2 \le n-2$, we have $\tilde \vv_{l,k,\eps}^{[j_1]}(x) = \tilde \vv_{l,k,\eps}^{(j_1)}(x)$ and $\tilde \eta_{l,k,\eps}^{[j_2]}(x)$ satisfies the relation analogous to \eqref{quasieta}: 
\begin{equation*} 
\tilde \eta_{l,k,\eps}^{[j_2]}(x) = \tilde \eta_{l,k,\eps}^{(j_2)}(x) + \sum_{j = 0}^{\nu - 2} \left( \sum_{s = j}^{\nu-2} (-1)^{s + \nu} C_s^j \tilde p_{n-\nu + s}^{(s-j)}(x)\right) \tilde \eta_{l,k,\eps}^{(j)}(x).
\end{equation*}

Therefore, using the induction by $j_1 + j_2 = 0, 1, \dots, n-2$, we get from Proposition~\ref{prop:series} that, for $j_1 + j_2 < n-2$, the series
$$
\tilde T_{j_1,j_2}(x) := \sum_{l = 1}^{\infty} \sum_{k = 1}^{n-1} (\tilde \vv_{l,k,0}^{(j_1)}(x) \tilde \eta_{l,k,0}^{(j_2)}(x) - \tilde \vv_{l,k,1}^{(j_1)} \tilde \eta_{l,k,1}^{(j_2)}(x))
$$
converges absolutely and uniformly on $[0,1]$ and, for $j_1 + j_2 = n-2$, it converges in $L_2[0,1]$ with regularization, in other words, the series
$$
\tilde T_{j_1,j_2}^{reg}(x) := \sum_{l = 1}^{\infty} \sum_{k = 1}^{n-1} (\tilde \vv_{l,k,0}^{(j_1)}(x) \tilde \eta_{l,k,0}^{(j_2)}(x) - \tilde \vv_{l,k,1}^{(j_1)}(x) \tilde \eta_{l,k,1}^{(j_2)}(x) - a_{j_1,j_2,l,k}) 
$$
converges in $L_2[0,1]$. The regularization constants $a_{j_1,j_2,l,k}$ are the same as in Proposition~\ref{prop:series}, because the difference $\mathscr T_{j_1,j_2}^{reg}(x) - \tilde T_{j_1,j_2}^{reg}(x)$ for $j_1 + j_2 = n-2$ can be represented as a linear combination of several series $T_{i_1,i_2}(x)$ with lower powers ($i_1 + i_2 < n - 2$), which converge absolutely and uniformly on $[0,1]$. Addition and subtraction of series throughout this proof are understood in the sense \eqref{sum}. Moreover, the corresponding estimates hold:
\begin{align} \label{estT1}
\max_{x \in [0,1]} | \tilde T_{j_1,j_2}(x) | \le C \Omega, & \quad j_1 + j_2 < n - 2, \\ \label{estT2}
\| \tilde T_{j_1,j_2}^{reg}(x) \|_{L_2[0,1]} \le C \Omega, & \quad j_1 + j_2 = n-2.
\end{align}

Consider the series
\begin{align*}
T_{j_1,j_2}(x) - \tilde T_{j_1, j_2}(x) = & \sum_{(l,k,\eps) \in V} (\vv_{l,k,\eps}^{(j_1)}(x) - \tilde \vv_{l,k,\eps}^{(j_1)}(x)) \tilde \eta_{l,k,\eps}^{(j_2)}(x) \\
 = & \sum_{l = 1}^{\infty} \sum_{k = 1}^{n-1} (\vv_{l,k,0}^{(j_1)}(x) - \vv_{l,k,1}^{(j_1)}(x) - \tilde \vv_{l,k,0}^{(j_1)}(x) + \tilde \vv_{l,k,1}^{(j_1)}(x)) \tilde \eta_{l,k,0}^{(j_2)}(x) \\
 & + \sum_{l = 1}^{\infty} \sum_{k = 1}^{n-1} (\vv_{l,k,1}^{(j_1)}(x) - \tilde \vv_{l,k,1}^{(j_1)}(x)) (\tilde \eta_{l,k,0}^{(j_2)}(x) - \tilde \eta_{l,k,1}^{(j_2)}(x)).
\end{align*}

Let us apply the estimates of Lemmas~\ref{lem:vv} and~\ref{lem:eta} for $\vv_{l,k,\eps}^{(j_1)}(x)$ and $\tilde \eta_{l,k,\eps}^{(j_2)}(x)$, respectively. We have the two cases:
\begin{align*}
j_1 = 0 \colon \quad & |T_{j_1,j_2}(x) - \tilde T_{j_1,j_2}(x)| \le C \Omega \sum_{l = 1}^{\infty} l^{j_2} \xi_l \chi_l, \\
j_1 > 0 \colon \quad & |T_{j_1,j_2}(x) - \tilde T_{j_1,j_2}(x)| \le C \Omega \sum_{l = 1}^{\infty} l^{j_1 + j_2 - 1} \xi_l.
\end{align*}

Since $\{ l^{n-2} \xi_l \} \in l_2$ and $\{ \chi_l \} \in l_2$, we conclude that the series $T_{j_1,j_2}(x) - \tilde T_{j_1,j_2}(x)$ in both cases converges absolutely and uniformly on $[0,1]$ for $j_1 + j_2 \le n-2$. Moreover, for the difference $T_{j_1,j_2}(x) - \tilde T_{j_1,j_2}(x)$, the estimates similar to \eqref{estT1} and \eqref{estT2} hold. For $j_1 + j_2 = n-2$, this readily implies the convergences of the series
$$
T_{j_1,j_2}^{reg}(x) = \tilde T_{j_1,j_2}^{reg}(x) + (T_{j_1,j_2}(x) - \tilde T_{j_1,j_2}(x))
$$
in $L_2[0,1]$ and the estimate \eqref{estTreg}.

Now, let $j_1 + j_2 = n - 3$. Formal differentiation together with the summation rule \eqref{sum} imply 
$$
T'_{j_1,j_2}(x) = T_{j_1+1,j_2}(x) + T_{j_1,j_2 + 1}(x).
$$
As we have already shown, the series $T_{j_1+1,j_2}(x)$ and $T_{j_1,j_2 + 1}(x)$ converge in $L_2[0,1]$ with regularization and satisfy the estimate \eqref{estTreg}. Moreover, for their sum, the regularization constants vanish: $a_{j_1 + 1, j_2,l,k} + a_{j_1,j_2 + 1, l,k} = 0$, so the series $T'_{j_1,j_2}(x)$ converges in $L_2[0,1]$ without regularization. Consequently, $T_{j_1,j_2} \in W_2^1[0,1]$ and
\begin{align*}
\| T_{j_1,j_2}(x) \|_{W_2^1[0,1]} & \le \| T_{j_1,j_2}(x) \|_{L_2[0,1]} + \| T_{j_1,j_2}'(x) \|_{L_2[0,1]} \\ & \le \max_{x \in [0,1]} |T_{j_1,j_2}(x)| + \| T_{j_1+1,j_2}^{reg}(x) \|_{L_2[0,1]} + \| T_{j_1, j_2 + 1}^{reg}(x) \|_{L_2[0,1]} \le C \Omega. 
\end{align*}

Thus, the lemma is already proved for $j_1 + j_2 = n-2$ and $j_1 + j_2 = n-3$. By induction, we complete the proof for $j_1 + j_2 = n-4, \dots, 1, 0$.
\end{proof}

Thus, we are ready to investigate the convergence of the series in the reconstruction formulas \eqref{recp}.

\begin{lem} \label{lem:convp}
The reconstruction formulas \eqref{recp} define the functions $p_s$ of the corresponding spaces $W_2^{s-1}[0,1]$ for $s = \overline{0,n-2}$ and
$\| p_s(x) - \tilde p_s(x) \|_{W_2^{s-1}[0,1]} \le C \Omega$.
\end{lem}

\begin{proof}
We prove the lemma by induction. Fix $s \in \{ 0, 1, \dots, n-2 \}$. Suppose that the assertion of the lemma is already proved for all $s_1 > s$.
Due to \eqref{deft}, we have
\begin{align*}
\mathscr S_1(x) & := t_{n,s}(x) + (-1)^{n-s} T_{0,n-s-1}(x) = \sum_{j = 0}^{n-s-1} b_j T_{n-s-1-j, j}(x) \\ & = \frac{d}{dx}\left( \sum_{j = 0}^{n-s-2} d_j T_{n-s-2-j,j}(x) \right),
\end{align*}
where
\begin{gather*}
b_j := C_n^{j+s+1} C_{j+s}^s + (-1)^{n-s} \de_{j,n-s-1}, \quad \sum\limits_{j = 0}^{n-s-1} b_j = 0, \\
d_j := \sum_{i = 0}^j (-1)^{j-i} b_i, \quad j = \overline{0,n-s-2}.
\end{gather*}

By virtue of Lemma~\ref{lem:Tconv}, the series $T_{n-s-2-j,j}$ for $j = \overline{0,n-s-2}$ belong to $W_2^s[0,1]$ for $s \ge 1$ and converge with regularization in $L_2[0,1]$ for $s = 0$. In both cases, we conclude that $\mathscr S_1 \in W_2^{s-1}[0,1]$. Furthermore, $\| \mathscr S_1(x)\|_{W_2^{s-1}[0,1]} \le C \Omega$.

Consider the next term in \eqref{recp}:
$$
\mathscr S_2(x) := \sum_{j = 0}^{n-s-3} \sum_{r = j}^{n-s-3} (-1)^r C_r^j \tilde p_{r+s+1}^{(r-j)}(x) T_{0,j}(x).
$$

In view of $p_k \in W_2^{k-1}[0,1]$ and Lemma~\ref{lem:Tconv}, we have
\begin{align*}
& \tilde p_{r+s+1}^{(r-j)} \in W_2^{s+j}[0,1] \subseteq W_2^s[0,1], \\
& T_{0,j} \in W_2^{n-j-2}[0,1] \subseteq W_2^{s+1}[0,1].
\end{align*}
Hence $\mathscr S_2 \in W_2^s[0,1]$. 

For the last term
$$
\mathscr S_3(x) := -\sum_{r = s+1}^{n-2} p_r(x) t_{r,s}(x),
$$
we have $p_r \in W_2^{r-1}[0,1]$ by the induction hypothesis, so $p_r \in W_2^s[0,1]$, and $t_{r,s} \in W_2^{s+1}[0,1]$ according to \eqref{deft} and Lemma~\ref{lem:Tconv}. Hence $\mathscr S_3 \in W_2^s[0,1]$. Lemma~\ref{lem:Tconv} also implies $\| \mathscr S_j(x) \|_{W_2^s[0,1]} \le C \Omega$ for $j = 2, 3$.

Since $p_s(x) = \tilde p_s(x) + \mathscr S_1(x) + \mathscr S_2(x) + \mathscr S_3(x)$, we arrive at the assertion of the lemma.
\end{proof}

Thus, we have obtained the vector $p = (p_k)_{k = 0}^{n-2}$ by the reconstruction formulas. It remains to prove the following lemma.

\begin{lem} \label{lem:sd}
The spectral data of $p$ coincide with $\{ \la_{l,k}, \be_{l,k} \}_{l \ge 1, \, k = \overline{1,n-1}}$.
\end{lem}

\begin{proof}
The proof is based on the approximation approach which has been considered in \cite{Bond23-res} for $n = 3$ in detail. The proof for higher orders $n$ is similar, so we omit the technical details and just outline the main idea.

Along with $\{ \la_{l,k}, \be_{l,k} \}_{l \ge 1, \, k = \overline{1,n-1}}$, consider the ``truncated'' data
$$
\la_{l,k}^N := \begin{cases}
                    \la_{l,k}, & l \le N, \\
                    \tilde \la_{l,k}, & l > N,
               \end{cases} 
\qquad
\be_{l,k}^N := \begin{cases}
                    \be_{l,k}, & l \le N, \\
                    \tilde \be_{l,k}, & l > N.
               \end{cases} 
$$

Then, by using $\{ \la_{l,k}^N, \be_{l,k}^N \}$ instead of $\{ \la_{l,k}, \be_{l,k} \}$, one can construct the main equation 
\begin{equation} \label{mainN}
(I - \tilde R^N(x)) \psi^N(x) = \tilde \psi^N(x), \quad x \in [0,1],
\end{equation}
analogously to \eqref{main}. It can be shown that equation \eqref{mainN} is uniquely solvable for sufficiently large values of $N$. Therefore, one can find the functions $\vv_{l,k,\eps}^N(x)$, $(l,k,\eps) \in V$, by using the solution $\psi^N(x)$ similarly to \eqref{findvv}. Then, construct the functions $\{ \Phi_k^N(x, \la) \}_{k = 1}^n$ and $p^N = (p_k^N)_{k = 0}^{n-2}$ analogously to \eqref{Phik0} and \eqref{recp}, respectively. The advantage of the ``truncated'' data $\{ \la_{l,k}^N, \be_{l,k}^N \}$ over $\{ \la_{l,k}, \be_{l,k} \}$ is that the series in \eqref{Phik0} and \eqref{recp} are finite, so one can show by direct calculations that $\{ \Phi_k^N(x, \la) \}_{k = 1}^n$ are the Weyl solutions of equation \eqref{eqv} with the coefficients $p^N$ and deduce that $\{ \la_{l,k}^N, \be_{l,k}^N \}$ are the spectral data of $p^N$. At this stage, the assumptions (S-1)--(S-3) of Theorem~\ref{thm:glob} are crucial. If these assumptions do not hold, then one needs additional data to recover $p^N$. Next, using \eqref{recp}, we show that 
$$
\lim_{N \to \infty} \| p_k^N - p_k \|_{W_2^{k-1}[0,1]} = 0, \quad k = \overline{0,n-2}. 
$$
Furthermore, it can be shown that the spectral data depend continuously on the coefficients $p = (p_k)_{k = 0}^{n-2}$. This concludes the proof.
\end{proof}

Let us summarize the arguments of this section in the proof of Theorem~\ref{thm:glob}.

\begin{proof}[Proof of Theorem~\ref{thm:glob}]
Let $\{ \la_{l,k}, \be_{l,k} \}$ and $\tilde p$ satisfy the hypothesis of the theorem. Then, the main equation \eqref{main} is uniquely solvable. By using its solution $\psi(x)$, we find the functions $\{ \vv_{l,k,\eps}(x) \}_{(l,k,\eps) \in V}$ by \eqref{findvv} and reconstruct the coefficients $(p_k)_{k = 0}^{n-2}$ by \eqref{recp}. By virtue of Lemma~\ref{lem:convp}, $p_k \in W_2^{k-1}[0,1]$ for $k = \overline{0,n-2}$. Lemma~\ref{lem:sd} implies that $\{ \la_{l,k}, \be_{l,k} \}$ are the spectral data of $p = (p_k)_{k = 0}^{n-2}$. Taking (S-1) and (S-2) into account, we conclude that $p \in W$. The uniqueness of $p$ is given by Corollary~\ref{cor:uniq}.
\end{proof}

\begin{proof}[Proof of Theorem~\ref{thm:loc}]
Let us show that, if data $\{ \la_{l,k}, \be_{l,k} \}_{l \ge 1, \, k = \overline{1,n}}$ and $\tilde p$ satisfy the conditions of Theorem~\ref{thm:loc} for sufficiently small $\de > 0$, then they also satisfy the hypothesis of Theorem~\ref{thm:glob}.

It follows from \eqref{condloc} that
$$
|\la_{l,k} - \tilde \la_{l,k}| \le \de, \quad |\be_{l,k} - \tilde \be_{l,k}| \le \de, \quad l \ge 1, \, k = \overline{1,n-1}.
$$

On the other hand, Definition~\ref{def:W} and the asymptotics \eqref{asymptla} and \eqref{asymptbe} imply that the eigenvalues $\{ \tilde \la_{l,k} \}_{l \ge 1}$ are separated for each fixed $k \in \{ 1, 2, \dots, n-1\}$ as well as for neighboring values of $k$ and the weight numbers $\{ \tilde \be_{l,k} \}$ are separated from zero. Rigorously speaking, we have
\begin{align*}
& |\tilde \la_{l,k} - \tilde \la_{l_0,k}| \ge \de_0, \quad l_0 \ne l, \quad k = \overline{1,n-1}, \\
& |\tilde \la_{l_0,k} - \tilde \la_{l,k+1}| \ge \de_0, \quad l,l_0 \ge 1, \quad k = \overline{1,n-2}, \\
& |\tilde \be_{l,k}| \ge \de_0, \quad l \ge 1, \quad k = \overline{1,n-1}.
\end{align*}

By choosing $\de < \de_0/2$, we achieve the conditions (S-1)--(S-3) of Theorem~\ref{thm:glob} for $\{ \la_{l,k}, \be_{l,k} \}_{l \ge 1, \, k = \overline{1,n-1}}$. Furthermore, (S-4) directly follows from \eqref{defxi} and \eqref{condloc}.

Using \eqref{estpsiR} and \eqref{defOmega}, we estimate 
$$
\| \tilde R(x) \|_{m \to m} = \sup_{v_0 \in V} \sum_{v \in V} |\tilde R_{v_0,v}(x)| \le C \sup_{l_0 \ge 1} \sum_{l = 1}^{\infty} \frac{\xi_l}{|l - l_0| + 1} \le C \Omega,
$$
where the constant $C$ depends only on $\tilde p$ and $\delta$ if $\Omega \le \delta$. Therefore, choosing a sufficiently small $\delta$ for the fixed $\tilde p$, we achieve $\| \tilde R(x) \| \le \frac{1}{2}$. Then, the operator $(I - \tilde R(x))$ has a bounded inverse, that is, the condition (S-5) of Theorem~\ref{thm:glob} is fulfilled.

Thus, the numbers $\{ \la_{l,k}, \be_{l,k} \}_{l \ge 1, \, k = \overline{1,n-1}}$ satisfy the conditions (S-1)--(S-5) of Theorem~\ref{thm:glob}, which implies the existence of $p = (p_k)_{k = 0}^{n-2} \in W$ with the spectral data $\{ \la_{l,k}, \be_{l,k} \}$. It is worth noting that, if $\Omega \le \delta$, then the constant $C$ in Lemmas~\ref{lem:vv}--\ref{lem:convp} depends only on $\tilde p$ and $\de$. Hence, the estimate \eqref{estp1} follows from Lemma~\ref{lem:convp}. This concludes the proof.
\end{proof}

\medskip

\section{Conclusion} \label{sec:concl}

In this paper, we have proved local solvability and stability of recovering the coefficients $(p_k)_{k = 0}^{n-2}$ of equation \eqref{eqv} from the spectral data, which consist of the eigenvalues $\{ \la_{l,k} \}_{l \ge 1}$ and the weight numbers $\{ \be_{l,k} \}_{l \ge 1}$ of the boundary value problems $\mathcal L_k$, $k = \overline{1,n-1}$. Moreover, we have obtained some sufficient conditions for the global solution of the inverse problem (Theorem~\ref{thm:glob}). The proof method is constructive. It is based on the reduction of the inverse problem to the linear main equation \eqref{main} in the Banach space $m$ and on the reconstruction formulas \eqref{recp} for the coefficients $p_k$, $k = \overline{0,n-2}$. We have proved the convergence of the series from the reconstruction formulas in the corresponding spaces $W_2^{k-1}[0,1]$, $k = \overline{0,n-2}$. Theorem~\ref{thm:loc} on local solvability and stability generalizes the analogous results of \cite{HM-sd} for $n = 2$ and of \cite{Bond23-res} for $n = 3$.

Our results have the following \textit{advantages} over the previous studies:
\begin{enumerate}
\item We for the first time proved local solvability and stability of an inverse spectral problem for differential operators with distribution coefficients of orders $n \ge 4$.
\item Our method is constructive. Basing on Procedure~\ref{alg:1}, one can develop a numerical method for solving the inverse problem.
\item Our method does not require self-adjointness. All the results of this paper are obtained for the general non-self-adjoint case.
\end{enumerate}

The results and the methods of this paper can be used for future development of spectral theory for higher-order differential operators with regular and distribution coefficients. Let us mention some \textit{open problems} in this direction.

\begin{enumerate}
\item In the method of spectral mappings, the unique solvability of the main equation plays an important role. In this paper, we have used the smallness of the spectral data perturbation (i.e. the smallness of $\de$ in Theorem~\ref{thm:loc}) to guarantee the existence of the main equation solution. It is worth considering other cases, when the unique solvability of the main equation can be proved. In particular, for $n = 2$ and $n = 3$, the main equation is uniquely solvable in the self-adjoint case. For differential operators on a finite interval for $n \ge 4$, this issue is open.
\item Using the technique of \cite{Bond22-asympt}, the following precise asymptotics can be obtained for the spectral data:
\begin{align*}
\la_{l,k} & = l^n \left(c_{0,k} + c_{1,k} l^{-1} + c_{2,k} l^{-2} + \dots + c_{n-1,k} l^{-(n-1)} + l^{-(n-1)} \varkappa_{l,k} \right), \\
\be_{l,k} & = l^n \left(d_{0,k} + d_{1,k} l^{-1} + d_{2,k} l^{-2} + \dots + d_{n-2,k} l^{-(n-2)} + l^{-(n-2)} \varkappa_{l,k}^0\right),
\end{align*}
where $c_{j,k}$, $d_{j,k}$ are constants and $\{ \varkappa_{l,k} \}$, $\{ \varkappa_{l,k}^0 \}$ are $l_2$-sequences. Hence, if $c_{j,k} = \tilde c_{j,k}$ and $d_{j,k} = \tilde d_{j,k}$, then $\{ l^{n-2} \xi_l \} \in l_2$, that is, the assumption (S-4) of Theorem~\ref{thm:glob} holds. Thus, in order to achieve (S-4), one has to construct a model problem with the known coefficients $c_{j,k}$ and $d_{j,k}$ in the spectral data asymptotics. The construction of such problem will help to finalize the global solvability results for higher-order inverse problems.
\item Remove the conditions (W-2) and (W-3) of Definition~\ref{def:W} (i.e. the simplicity of the spectra and the separation condition). Then, other spectral data are required and the problem becomes more technically complicated. In addition, one has to take splitting of multiple eigenvalues under small perturbations of the spectra into account (see \cite{BK19} for $n = 2$).
\item Obtain solvability conditions for differential expressions with coefficients of higher singularity orders than $p_k \in W_2^{k-1}[0,1]$. Although the uniqueness theorems in \cite{Bond21, Bond23-mmas, Bond23-reg} and the reconstruction approach \cite{Bond22-alg} are obtained for coefficients of wider distribution spaces, it is a challenge to prove the existence of the inverse problem solution.
\item Investigate the uniform stability of inverse problems for higher-order differential operators (see \cite{SS10, Hryn11} for $n = 2$).
\item Study the reconstruction of higher-order differential operators from the finite spectral data $\{ \la_{l,k}, \be_{l,k} \}_{l = 1}^N$. For $n = 2$, see, e.g., \cite{MW05, SS14}.
\end{enumerate}

\medskip

{\bf Funding.} This work was supported by Grant 21-71-10001 of the Russian Science Foundation, https://rscf.ru/en/project/21-71-10001/ (accessed on 13 August 2023).

\noindent Natalia Pavlovna Bondarenko \\

\noindent 1. Department of Mechanics and Mathematics, Saratov State University, \\
Astrakhanskaya 83, Saratov 410012, Russia, \\

\noindent 2. Department of Applied Mathematics and Physics, Samara National Research University, \\
Moskovskoye Shosse 34, Samara 443086, Russia, \\

\noindent 3. Peoples' Friendship University of Russia (RUDN University), \\
6 Miklukho-Maklaya Street, Moscow, 117198, Russia, \\

\noindent e-mail: {\it bondarenkonp@info.sgu.ru}

\end{document}